\documentclass[11pt]{article}

\usepackage[utf8]{inputenc}
\usepackage{amssymb, amsmath, amsthm}
\usepackage{relsize}
\usepackage{verbatim}
\usepackage[a4paper, total={155mm, 230mm}]{geometry}
\usepackage{cite}
\usepackage{hyperref}
\usepackage{booktabs}
\usepackage{multirow}
\usepackage{graphicx}
\usepackage[cmtip,all]{xy}

\newtheorem{theorem}{Theorem}[section]
\newtheorem{lemma}[theorem]{Lemma}
\newtheorem{conjecture}[theorem]{Conjecture}
\newtheorem{problem}[theorem]{Open Problem}
\newtheorem{proposition}[theorem]{Proposition}
\newtheorem{corollary}[theorem]{Corollary}

\theoremstyle{definition}
\newtheorem{definition}[theorem]{Definition}
\newtheorem{example}[theorem]{Example}
\newtheorem{exercise}{Exercise}[section]

\theoremstyle{remark}
\newtheorem{remark}[theorem]{Remark}

\newcommand{\bfb}{\ensuremath{\mathbf{b}}}
\newcommand{\bfc}{\ensuremath{\mathbf{c}}}
\newcommand{\bfm}{\ensuremath{\mathbf{m}}}
\newcommand{\bfn}{\ensuremath{\mathbf{n}}}

\newcommand{\bfv}{\ensuremath{\mathbf{v}}}
\newcommand{\bfw}{\ensuremath{\mathbf{w}}}
\newcommand{\bfx}{\ensuremath{\mathbf{x}}}
\newcommand{\bfy}{\ensuremath{\mathbf{y}}}

\newcommand{\bbP}{\ensuremath{\mathbb{P}}}
\newcommand{\bbF}{\ensuremath{\mathbb{F}}}
\newcommand{\bbK}{\ensuremath{\mathbb{K}}}
\newcommand{\bbA}{\ensuremath{\mathbb{A}}}
\newcommand{\bbZ}{\ensuremath{\mathbb{Z}}}
\newcommand{\bbC}{\ensuremath{\mathbb{C}}}
\newcommand{\bbQ}{\ensuremath{\mathbb{Q}}}
\newcommand{\bbR}{\ensuremath{\mathbb{R}}}
\newcommand{\bbN}{{\mathbb{N}}}

\newcommand{\fm}{\ensuremath{\mathfrak{m}}}

\newcommand{\cL}{\ensuremath{\mathcal{L}}}
\newcommand{\cI}{\ensuremath{\mathcal{I}}}
\newcommand{\cO}{\ensuremath{\mathcal{O}}}
\newcommand{\cT}{\ensuremath{\mathcal{T}}}
\newcommand{\cQ}{{\mathcal{Q}}}
\newcommand{\cD}{{\mathcal{D}}}

\newcommand{\cB}{{\mathcal{B}}}

\newcommand{\Pic}{\ensuremath{\operatorname {Pic}}}
\newcommand{\NE}{\ensuremath {\operatorname {NE}}}
\newcommand{\NS}{\ensuremath {\operatorname {NS}}}
\newcommand{\Nef}{\ensuremath{\operatorname {Nef}}}
\newcommand{\Mor}{\ensuremath{\overline{\operatorname {NE}}}}
\newcommand{\cent}{\ensuremath{\operatorname {center}}}
\newcommand{\Spec}{\ensuremath{\operatorname {Spec}}}
\newcommand{\Proj}{\ensuremath{\operatorname {Proj}}}
\newcommand{\ord}{\ensuremath{\operatorname {ord}}}
\newcommand{\mult}{\ensuremath{\operatorname {mult}}}
\newcommand{\vol}{\ensuremath{\operatorname {vol}}}
\newcommand{\co}{\ensuremath{\operatorname {co}}}
\newcommand{\Div}{\ensuremath{\operatorname {Div}}}
\newcommand{\CDiv}{\ensuremath{\operatorname {CDiv}}}
\newcommand{\Eff}{\ensuremath{\operatorname {Eff}}}
\newcommand{\trdeg}{\ensuremath{\operatorname{tr.deg}}}
\newcommand{\GL}{\ensuremath{\operatorname {GL}}}

\makeatletter
\def\vv{\@ifnextchar[{\@withv}{\@withoutv}}
\def\@withv[#1]#2#3{v_{#2,#3}(#1)}
\def\@withoutv#1#2{v_{#1,#2}}
\makeatother

\makeatletter
\def\mm{\@ifnextchar[{\@withm}{\@withoutm}}
\def\@withm[#1]#2#3{\mu_{#1}(#2,#3)}
\def\@withoutm#1#2{\widehat\mu(#1,#2)}
\makeatother

\title{Nagata type statements}
\author{Joaquim Roé and Paola Supino}
\date{}

\begin{document}
\maketitle
\begin{abstract}
Nagata solved Hilbert's 14-th problem in 1958
in the negative. The solution naturally lead him
to a tantalizing conjecture that remains widely open
after more than half a century of intense efforts.
Using Nagata's theorem as starting point, and the
conjecture, with its multiple variations, as motivation,
we explore the important questions of finite generation
for invariant rings, for support semigroups of multigraded
algebras, and for Mori cones of divisors on blown up surfaces,
and the rationality of Waldschimdt constants.
Finally we suggest a connection between the Mori cone
of the Zariski--Riemann space and the continuity of
the Waldschmidt constant as a function on the space
of valuations.

These notes correspond to the course of
the same title given by the first author in the workshop
``Asymptotic invariants attached to linear series''
held in the Pedagogical University of Cracow from May 16 to 20, 2016.
\end{abstract}

\tableofcontents

\maketitle

\section*{Introduction}
Hilbert's 14-th problem on finite generation of   algebras that are invariant under the action of some groups was formulated in the middle of La Belle \'{E}poque as an algebraic question. A few decades later, just when pop art was sprouting and rock--and--roll  music was turning to the surf--rock music, Zariski translated it into a geometric counterpart asking when the total coordinate ring of a projective variety is finitely generated, and Nagata gave a negative answer to it producing certain blowups of projective spaces, which couldn't but spur a renewed interest on the subject.

 The first section of these notes is devoted to Nagata's results
 from a modern point of view, taking into account contributions by
 Mukai, Ciliberto-Miranda and Ciliberto-Harbourne-Miranda-Roé.
 We show how one can construct  a group $G$, associated to $n$ points of
 the complex projective plane with multiplicities (a fat point scheme).
 An  action of $G$  on the  ring of polynomials in $2n$ indeterminates is
 then given, such that the algebra of invariants of $G$ is the Rees algebra
 of the ideal $I$ of the scheme of points: it is the direct sum of all the symbolic powers of $I$, thus, it is naturally a bigraded algebra.
 Its support, that is,  the subset of indices such that the corresponding
 addend is not trivial, is a semigroup. In the case that  this semigroup
 turns out not to be  finitely generated, then the same holds for the
 algebra.
  One can study the  real convex cone spanned by the semigroup: as Nagata observed, if it is not closed, then the semigroup, and hence the algebra, cannot be finitely generated. For suitable choices of the points and their multiplicities, this is exactly the case.

The second section is devoted to the Mori cone of
curves on the blowup of the $n$ points, following the work
of Waldschmidt, Demailly, Harbourne, de Fernex and
Ciliberto-Harbourne-Miranda-Roé.
We describe how the real cone of the first section
can be understood as a slice of the Mori cone,
and then Nagata's conjecture
can be interpreted as a statement on the boundary of this cone.
Numerical invariants such as Waldschmidt constants or Seshadri
constants, which control slopes of certain extremal rays
in the Mori cone, then come into play,
leading to the question of existence of irrational
Waldschmidt and Seshadri constants
and to the quest for extremal rays in the Mori cone.

Analogous statements can be made considering valuations as
generalizations of points. This point of view
was initiated by Dumnicki-Harbourne-Küronya-Roé-Szemberg,
and the last two sections are devoted to this subject.
It leads to
conjectures that make sense for real values
$t\geq 1$ of the number of points rather than integral ones,
and to the study of cones of effective \emph{b--divisors} on the
Zariski--Riemann space of the projective plane.

\subsection*{Acknowledgements}
We warmly thank the organizers of the workshop
``Asymptotic invariants attached to linear series'' in Cracow in May, 2016, which gave to us the opportunity of working in a friendly stay, and all the participants  for stimulating discussions. In particular we thank B.~Harbourne for sharing the notes \cite{Har16} of his course,
available in this volume.
We are also grateful to the Simons Foundation, MNiSW and IM PAN for financial support. Joaquim Roé was
partially supported by MTM 2013-40680-P (Spanish MICINN grant)
and 2014 SGR 114 (Catalan AGAUR grant).
\section{Nagata's theorem and conjecture}

\subsection{Nagata's Theorem}
\subsubsection*{Hilbert's 14-th problem}
Let $k$ be a field, let $z_1,\ldots, z_{\mu}$ be indeterminates over $k$ and let
$\bbK$ be an \emph{intermediate field} between $k$ and $k(z_1,\ldots, z_{\mu})$, i.e.
\[ k\;\subseteq\; \bbK\; \subseteq\; k(z_1,\ldots, z_{\mu}).\]

Hilbert's 14-th problem asks: \emph{is $\bbK\cap k[z_1,\ldots, z_{\mu}]$ a finitely generated $k$-algebra?}

Hilbert had in mind the following situation coming from invariant theory. Let $G$ be a subgroup of the \emph{affine group}, i.e. the group of automorphisms of  $\bbA_k^{\mu}$. Then $G$ acts as a set of automorphisms of the $k$-algebra $k[z_1,\ldots, z_{\mu}]$, hence on $k(z_1,\ldots, z_{\mu})$,   and we let $\bbK=k(z_1,\ldots, z_{\mu})^G$ be the field of $G$-invariant elements. Then the question is: \emph {is
\[ k[z_1,\ldots, z_{\mu}]^G=\bbK \cap k[z_1,\ldots, z_{\mu}]\]
a finitely generated $k$-algebra?}

In the case ${\mu}=1$, Hilbert's problem has trivially an
affirmative answer.
The answer is also affirmative for ${\mu}=2$, as proved by Zariski in \cite{Zar54}.
In \cite{Nag58}, \cite {Nag59}, Nagata provided counterexamples to the latter
formulation of Hilbert's problem. Nagata's \emph{minimal counterexample}
has ${\mu}=32$ and $\trdeg(\bbK/k)=4$.
Several other counterexamples have been given by various authors, too long a
story to be reported on here. The most recent one is due to Totaro  \cite {Tot08}, who
shows that Nagata's construction and some of its variations work even
over a finite field $k$.

We now give a streamlined review of Nagata's counterexample drawing
on the more general constructions of Mukai \cite{Muk01} and
Ciliberto--Harbourne--Miranda--Ro\'e \cite{CHMR13}.
For the sake of simplicity, we fix the base field to
be the complex numbers, $k=\bbC$.
\medskip

\subsubsection*{Nagata's group action}
  Let $P=(p_{ij})_{1\le i\le 3;1\le j\le n}$ be a $3\times n$ matrix
  of complex numbers.
  Its columns determine $n$ points $p_1,\ldots, p_n$ in
  $\bbP^{2}$; we will assume that they are $n$ distinct points,
  not all on a hyperplane (in particular, $\operatorname{rank} (P)=3$).
 The $(n-3)$--dimensional linear subspace $K=\ker (P)$ of $\bbC^n$
 formed by all vectors $\bfb=(b_1,\ldots,b_n)$ such that
$P\cdot \bfb={\bf 0}$,   is said to be \emph{associated} to $p_1,\ldots, p_n$.

Let $\bfx=(x_1,\ldots, x_n)$ and $\bfy=(y_1,\ldots, y_n)$ be vectors of indeterminates,
and consider the polynomial ring $\bbC[\bfx,\bfy]$ (so that ${\mu}=2n)$.
Initially \cite{Nag58}, Nagata considered the
\emph{unipotent} action of $K$ on $\bbC[\bfx,\bfy]$ given by
\begin{equation}
\begin{aligned}
\label{unipotent}
\bfb (x_i)&=x_i\\
\bfb (y_i)&=y_i+b_ix_i, \;  {\text {for}}\; 1\le i\le n.
\end{aligned}
\end{equation}
For adequate choices of $n$ and $P$
the $\bbC$-algebra $\bbC[\bfx,\bfy]^G$ is not finitely
generated, as we shall see.

Later, in \cite{Nag59}, with the goal of obtaining examples with smaller
transcendence degree, Nagata considered the action of a larger group,
which we now introduce in the generalized form of \cite{CHMR13}.
Fix  a  vector $\bfv=(v_1,\ldots, v_n)$ of positive integers
(``multiplicities'') and consider the following subgroup of the multiplicative group $(\bbC^*)^n$:
\[
H_\bfv=\{(c_1,\ldots,c_n)\in (\bbC^*)^n \,|\,c_1^{v_1} \cdots c_n^{v_n}=1\}.
\]
Given $\bfc \in H_\bfv$ and $\bfb \in K$, set
\begin{align}\label{biggroup}
\begin{split}
 \sigma_{\bfc,\bfb} (x_i)&=c_i\,x_i,\\
\sigma_{\bfc,\bfb} (y_i)&=\frac{c_i}{c_1\cdots c_n}(y_i+b_ix_i), \;
{\text {for}}\; 1\le i\le n.
\end{split}
\end{align}
This defines a semidirect product $G=H_\bfv\ltimes K$,
a $(2n-4)$--dimensional subgroup of
$(\bbC^*)^n\ltimes \bbC^n$, acting linearly on
$\bbC[\bfx,\bfy]$; here $\sigma_{\bfc,\bfb}$ is
the image in $\GL_{2n}(\bbC)$ of an element in $G$.
We shall identify the groups $H_\bfv$ and $K$ with their
isomorphic images $H_\bfv\times\{0\}$ and
$\{1\}\times K$ in $G=H_\bfv\ltimes K$.

Again, for adequate choices of $n$, $\bfv$, and $P$,
the $\bbC$-algebra $\bbC[\bfx,\bfy]^G$ is not finitely
generated.

\begin{exercise}
The semidirect product $G$ is determined by
an action
$\phi:H_\bfv\rightarrow \operatorname{Aut}K$ of $H_\bfv$ on $K$.
Make  this action and the resulting product
$G=H_\bfv\ltimes_\phi K$ explicit.
Nagata in \cite{Nag59} considered the case
$v_1=\cdots=v_n=1$. Show that this leads to the trivial
action, and hence to the direct product $G=H_\bfv\times K$.
\end{exercise}

In order to prove non finite generation,
Nagata's key insight is to identify $k[\bfx,\bfy]^G$ with a
graded algebra built from plane geometry; the kind of algebra that
will be the main object of study in these notes.
The proof then proceeds in two steps. First,
sufficient conditions are found  for the algebra to be
non-finitely generated, expressible in terms of the existence of curves
in the projective plane with given degree and multiplicities at the points $p_j$.
The second step consists in actually showing that such sufficient
conditions are satisfied for adequate choices of $n$ and $\bfv$,
if $P$ is general enough.

The construction can be carried over
using a matrix $P$ with $r\ge 3$ rows, leading to other counterexamples
to Hilbert's 14-th problem
related to the geometry of projective $(r-1)$--space.
This generalization is due to Mukai, who used it in \cite{Muk01}
to show counterexamples where the group acting on
$\bbC[\bfx,\bfy]$ is $K\cong \bbC^k$ for any $k\ge3$.
It is not known whether there exist counterexamples
for the group $\bbC^2$, while there are none for $\bbC$ by
Weitzenböck's result \cite{Wei32}.

\subsubsection*{The invariant ring of the unipotent action as a Rees algebra}
To describe the connection with geometry,
let us fix some additional notation.
Choose coordinates $\bfw=(w_1,w_2,w_3)$ on $\bbP^2$, 
so that
$\mathbb{C}[\bfw]=\mathbb{C}[\bbP^2]$ is the homogeneous coordinate ring of
$\bbP^2_{\mathbb{C}}$, and call $I(p_j)\subset \mathbb{C}[\bbP^2]$ the
homogeneous ideal of the point $p_j=[p_{1j}, p_{2j}, p_{3j}]$
for $j=1,\dots,n$, where $p_j\neq p_h$ for $j\neq h$. For an arbitrary vector of multiplicities
$\bfm=(m_1,\dots,m_n)$, by  abuse of language, and consistently with the notation in \cite{Har16}, in the rest of the paper we denote by $Z_{\bfm}=\sum_{j=1}^n m_jp_j$
the 0-dimensional subscheme of
$ \bbP^2_{\mathbb{C}}$ (a \emph{fat points scheme}) determined
by the homogeneous ideal
\(
I(Z_\bfm)=\bigcap_{j=1}^n I(p_j)^{m_j}.
\)

For any homogeneous ideal $I$  in a given graded ring,
denote as customary $I_t$ its homogeneous component in degree $t$.

Since the monomials $x_j$ are invariant under the unipotent action
\eqref{unipotent} of the
associated space $K$, this action can be extended to the ring
$\mathbb{C}[\bfx,\bfy] [x_1^{-1},\ldots,x_n^{-1}]$.
Here $K$ acts by translation, so the invariant ring
$\mathbb{C}[\bfx,\bfy] [x_1^{-1},\ldots,x_n^{-1}]^{K}$
can be immediately computed: it is generated by
\[
\varpi_1=\sum_{j=1}^n p_{1j} y_j/x_j; \quad \varpi_2= \sum_{j=1}^n p_{2j} y_j/x_j;\quad \varpi_3= \sum_{j=1}^n p_{3j} y_j/x_j
\]
over the ring $\mathbb{C} [x_1^{\pm1},\ldots,x_n^{\pm1}]$.
The  elements
\begin{equation}
\label{defw}
\begin{aligned}
\varpi_1 \cdot x_1\cdots x_n&=
p_{1,1}y_1x_2\cdots x_n+p_{1,2}x_1y_2x_3\cdots x_n+
\dots+p_{1,n}x_1 \cdots x_{n-1}y_n,\\
\varpi_2 \cdot x_1\cdots x_n&=
p_{2,1}y_1x_2\cdots x_n+p_{2,2}x_1y_2x_3\cdots x_n+
\dots +p_{2,n}x_1 \cdots x_{n-1}y_n,\\
\varpi_3 \cdot x_1\cdots x_n&=
p_{3,1}y_1x_2\cdots x_n+p_{3,2}x_1y_2x_3\cdots x_n+
\dots+p_{3,n}x_1 \cdots x_{n-1}y_n
\end{aligned}
\end{equation}
 are independent linear combinations of the obviously
algebraically independent elements
\[y_1x_2\cdots x_n,\quad x_1y_2\cdots x_n,\quad
\ldots,\quad x_1 \cdots x_{n-1}y_n,\]
so they are algebraically independent;
we identify them with the coordinates $w_1, w_2, w_3$, i.e.,
\[w_1=\varpi_1 \cdot x_1\cdots x_n, \quad
w_2=\varpi_2 \cdot x_1\cdots x_n, \quad
w_3=\varpi_3 \cdot x_1\cdots x_n.\]
Since these elements belong to  $\mathbb{C}[\bfx,\bfy]$ and are invariant  under $K$,
they realize $\bbC[\bbP^2]=\bbC[w_1,w_2,w_3]$
as a subring of $\mathbb{C}[\bfx,\bfy]^{K}$.
Then
\[
(\mathbb{C}[\bfx,\bfy] [x_1^{-1},\ldots,x_n^{-1}])^{K}=
\bbC[\bbP^2] [x_1^{\pm 1},\ldots,x_n^{\pm 1}]
\]
and thus
\begin{equation}\label{Ginv}
\mathbb{C}[\bfx,\bfy]^{K}=
(\mathbb{C}[\bfx,\bfy] [x_1^{-1},\ldots,x_n^{-1}])^{K}\cap \mathbb{C}[\bfx,\bfy] =
\bbC[\bbP^2] [x_1^{\pm 1},\ldots,x_n^{\pm 1}]\cap \mathbb{C}[\bfx,\bfy].
\end{equation}

\begin{remark}\label{changevar}
The identification of the three forms in \eqref{defw} with
$w_1, w_2, w_3$, and the identification of the columns of
$P$ with points in $\bbP^2$ are mutually consistent with
respect to changes of variables.
Indeed, given an invertible matrix
$A\in\operatorname{GL}_3$, consider
$P'=AP$, which has
the same associated space $K$,
hence the same invariant ring $\bbC[\bfx,\bfy]^K$.
The new invariant elements
$\bfw'=(w'_1,w'_2,w'_3)$ generate the same
invariant subring $\bbC[\bbP^2]$, as they satisfy $\bfw'=A\bfw$.
\end{remark}

For $j=1,\ldots,n$, let
$V_j$ be the linear space of homogeneous elements of $\bbC[\bbP^2]$ of degree 1
that are divisible by $x_j$ (in $\bbC[\bfx,\bfy]$).
Equivalently, $V_j$ is the linear subspace of $\bbC[\bbP^2]_1$
formed by elements whose coefficient in the monomial $(x_1\cdots x_n)y_j/x_j$
vanishes.

\begin{lemma}
\label{ }
A degree $d$ homogeneous polynomial $F$ in
$\bbC[\bbP^2]$ vanishes  on $p_j$ with multiplicity
$m_j$ if and only if it  belongs to
\[[{(V_j)^{m_j}}]_d=[{(x_j)^{m_j}\cap \bbC[\bbP^2]}]_d.\]
Moreover, in this case $F/x_j^{m_j}$ is in $\mathbb{C}[\bfx,\bfy]^K$.
 \end{lemma}
\begin{proof}
For simplicity of notation we assume that $j=1$ and set $m_1=m$.
Start with $m=1$ and consider a homogeneous polynomial of degree $d$
\[F(w_1,w_2,w_3)=\sum_{a+b+c=d}\alpha_{abc}w_1^a w_2^b w_3^c.\]
Expanding all powers of the $w_i$s' using \eqref{defw},
we see that the only terms that are not divisible by $x_1$ add up to
\[ \sum_{a+b+c=d}\alpha_{abc}\,p_{1,1}^ap_{2,1}^bp_{3,1}^c
(y_1 x_2\cdots x_n)^d,\]
and obviously this vanishes if and only if $F(p_1)=\sum\alpha_{abc}\,p_{1,1}^ap_{2,1}^bp_{3,1}^c=0$. 
The fact that $F/x_1\in\bbC[\bfx,\bfy]^K$ is now immediate
by \eqref{Ginv}.

The argument for $m>1$ is analogous, but in order to make
the computation more transparent we will do a further
reduction, and assume that the first point is the coordinate point
$p_1=(1,0,0)$ in $\bbP^2$.
By remark \ref{changevar} this is not restrictive.

Now $w_2$ and $w_3$ are multiples of $x_1$ by \eqref{defw}
(since $p_{2,1}=p_{3,1}=0$), and they span $V_1$.
Therefore $[(V_1)^m]_d\subset [(x_1)^m \cap \bbC[\bbP^2]]_d$.

Conversely, assume that $F$ has multiplicity $e<m$ at $p_1$.
In terms of the expansion
$F(w_1,w_2,w_3)=\sum\alpha_{abc}w_1^a w_2^b w_3^c$,
this means there are nonvanishing terms
$\alpha_{d-e,b,c}w_1^a w_2^b w_3^c$ with $b+c=e$.
As each of these $w_2^b w_3^c$ is a multiple of $x_1^e$,
and $w_1=y_1x_2\cdots x_n$ modulo $x_1$,
the following equality holds modulo $x_1^{e+1}$:
\[
F(w_1,w_2,w_3)=\sum_{a+b+c=d}\alpha_{abc}w_1^a w_2^b w_3^c=
\sum_{b+c=e}\alpha_{d-e,b,c} (y_1 x_2 \cdots x_n)^{d-e} w_2^b w_3^c\ne 0,
\]
i.e., $F(w_1,w_2,w_3)$ is not equal to zero
modulo $x_1^{e+1}$, so it is not divisible by $x_1^m$,
and we have proved the inclusion $[(V_1)^m]_d\supset [(x_1)^m \cap \bbC[\bbP^2]]_d$.
 \end{proof}

\begin{lemma}\label{fixedring}
Let as before $Z_{\bfm}=m_1p_1+\cdots+m_np_n\subset\bbP^2$; then
\begin{equation*}
    \mathbb{C}[\bfx,\bfy]^K \cong \underset{\bfm\in \bbZ^n}{\mathsmaller\bigoplus}
    I(Z_\bfm).
\end{equation*}
\end{lemma}
\begin{proof}
By \eqref{Ginv}, an element $f\in \bbC[\bfx,\bfy]$
is invariant by $K$ if and only if there exist nonnegative
integers $m_1, \dots, m_n$ such that
$fx_1^{m_1}\cdots x_n^{m_n}\in \bbC[\bbP^2]$; this
gives
\begin{equation*}
    \mathbb{C}[\bfx,\bfy]^K = \mathbb{C}[\bbP^2][x_1,\ldots ,x_n]+\sum_{\bfm> \underline{0}}((x_1)^{m_1}\cap\ldots \cap(x_n)^{m_n}\cap\bbC[\bbP^2])x_1^{-m_1 }\cdots x_n^{-m_n }.
\end{equation*}
The previous lemma then says that the last expression equals
\begin{equation*}
    \mathbb{C}[\bbP^2][x_1,\ldots ,x_n]+\sum_{\bfm> \underline{0}} I(Z_\bfm)x_1^{-m_1 }\cdots x_n^{-m_n },
\end{equation*}
that is clearly isomorphic to
$\oplus_{\bfm\in \bbZ^n} I(Z_\bfm)$, as claimed.
\end{proof}
The multigraded algebra of Lemma \ref{fixedring} is called the
\emph{Rees algebra of the multigraded filtration}
$\{I(Z_\bfm)\}_{\bfm\in\bbZ^n}$.
It also inherits the natural grading of $\bbC[\bbP^2]$,
so that it is in fact a $\bbZ^{n+1}$--graded algebra:
\[
\underset{\bfm\in \bbZ^n}{\mathsmaller{\bigoplus}}
    I(Z_\bfm)
    =\underset{{\bfm\in\bbZ^n, d\geq0}}{\mathsmaller{\bigoplus}}
    [I(Z_\bfm)]_d\, .
    \]
For details on Rees algebras for general filtrations, for modules,
and their connection with blowups, see \cite{GN94}, \cite{EHU03}.

\subsubsection*{The invariant ring of the Nagata action as a Rees algebra}
Let us now go back and
consider a fixed vector of multiplicities
$\bfv$, and the groups
$H_{\bfv}=\{(c_1,\ldots,c_n) \, | \, c_1^{v_1} \cdots c_n^{v_n}=1 \}$ and $G=H_\bfv \ltimes K$
acting by \eqref{biggroup}.
The algebra of invariants of $G$ can be
described as
\begin{equation}\label{iteratedaction}
\bbC[\bfx,\bfy]^G=\left(\bbC[\bfx,\bfy]^{K}\right)^G
=\left(\bbC[\bfx,\bfy]^K\right)^{H_\bfv}.
\end{equation}
The three elements $w_1, w_2, w_3$ are clearly invariant not only under the
action of $K$, but under the whole group $G$. Therefore $H_\bfv$ acts on
$\bbC[\bbP^2][x_1^{\pm1},\dots,x_n^{\pm1}]$, and in fact the action
can be described as follows. For every $\bfc\in H_\bfv$,
\begin{align*}
 \bfc (w_j)&=w_j, \;
{\text {for}}\; 1\le j\le 3, \\
 \bfc (x_i)&=c_ix_i, \;
{\text {for}}\; 1\le i\le n.
\end{align*}
Therefore, by the definition of $H_\bfv$,
\begin{equation}\label{diaginv}
\bbC[\bbP^2][x_1^{\pm1},\dots,x_n^{\pm1}]^{H_\bfv}=
\bbC[\bbP^2][t^{\pm1}], \; \text{ where }t=x_1^{v_1}\cdots x_n^{v_n}.
\end{equation}
For every  nonnegative integer $m$, let
\[
I(mZ_\bfv)=I(Z_\bfv)^{(m)}=I(Z_{m\bfv})=\bigcap_{j=1}^n I(p_j)^{mv_j}
\]
be the so-called $m$-th \emph{symbolic power} of $I(Z_\bfv)$.
Putting together \eqref{Ginv}, \eqref{iteratedaction}, \eqref{diaginv},
and Lemma \ref{fixedring}, the following description of the invariant ring
holds.

\begin{proposition}\label{fixedringbig}
Let $Z_\bfv=v_1p_1+\cdots+v_np_n\subset\bbP^2$; then
\begin{equation}\label{reebig}
    \mathbb{C}[\bfx,\bfy]^G \cong
    \underset{m\in \bbZ}{\mathsmaller{\bigoplus}}
    I(mZ_\bfv)
    =\underset{m\in\bbZ, d\geq 1}{\mathsmaller{\bigoplus}}
    [I(mZ_{\bfv})]_d.
\end{equation}
\end{proposition}

Again, we have identified the invariant ring as a Rees algebra.

\subsection{Semigroups, cones and finite generation}

The next step is to find sufficient conditions under which   the
 multigraded algebras  $\oplus_{\bfm,d}[I(Z_\bfm)]_d$ of
 Lemma \ref{fixedring}
 and  $\oplus_{m,d}[I(mZ_\bfv)]_d$ of Proposition \ref{fixedringbig} are not finitely
 generated.

Given a $k$-algebra $A=\oplus_{\lambda\in\Lambda}A_{\lambda}$ graded by a free abelian group $\Lambda$, the subset $\{\lambda\in\Lambda |A_{\lambda}\neq 0\}$ of $\Lambda$ is a semigroup   called the \emph{support of} $A$ and denoted by $Supp(A)$. Clearly, if $A$ is finitely generated as a ring over $k$ then $Supp(A )$  is finitely generated as a semigroup.
 In our case $\mathcal{S}_K=Supp(\oplus_{\bfm,d}[I(Z_\bfm)]_d)$
 (respectively $\mathcal{S}_G=Supp(\oplus_{m,d}[I(mZ_\bfv)]_d)$) is
 a semigroup in $\mathbb{Z}^{n+1}$ (respectively in $\bbZ^2$), and it
 will be enough to give conditions  in order that
$\mathcal{S}_K$ or $\mathcal{S}_G$ is not finitely generated.
In fact, we shall give sufficient conditions for the \emph{convex cone}
spanned by the semigroup $Supp A$ in the real vector space
 $ \Lambda \otimes \bbR\cong\bbR^{N}$ to be non finitely generated,
which is a stronger condition.

A \emph{convex cone} in a real vector space $V$ is a subset $C\subset V$
closed under nonnegative linear combinations:
\[ \forall u, v\in C, \forall a,b \in\bbR, a,b\ge 0 \Longrightarrow
au+bv \in C. \]
Given an arbitrary subset $S\subset V$, the \emph{cone spanned} by $S$
(or \emph{conic hull}) is the set of all nonnegative linear combinations
of vectors in $S$:
\[
\co(S)=\left\{\left.\sum_{i=1}^k a_i v_i \,\right|\,a_i\ge 0, v_i\in S \right\}.
\]
The conic hull $\co(v)$ of a nonzero vector is called the \emph{ray}
spanned by $v$. Given a cone $C$, a ray $R\subset C$ is said to be
\emph{extremal} if for every $u, v\in C$, $u+v\in R$ implies $u, v \in R$.
A cone is \emph{polyhedral} if it can be spanned by a finite set.
A polyhedral cone is always closed.

Given two cones $C_1, C_2$, the cone spanned by their union
is denoted by $C_1+C_2=\co(C_1\cup C_2)$, as it coincides with their
Minkowski sum as subsets of $V$.

Consider now the real convex cone
spanned by $Supp(A)$
 \[
 \co(Supp(A))=\left\{ \sum_{i=1}^k a_i \lambda_i \, | \,
 a_i \in \bbR_{\ge0}, \lambda_i \in Supp(A) \right\} \subset \Lambda \otimes \bbR\cong\bbR^{N}.
  \]
 Whenever the semigroup $Supp(A)$ is finitely generated,
 $\co(Supp(A))$ is a closed polyhedral cone, whose extremal rays are
 spanned by a subset of generators of $Supp(A)$.
Nagata's method to prove that
 $\mathcal S_G$ is not finitely generated is to show that $\co(\mathcal{S}_G)$
 is not closed. Observe that $\co(\mathcal{S}_G)$ can be understood as the intersection of
 $\co(\mathcal{S}_K)\subset \bbR^{n+1}$ with the plane
 \[ \Pi=\langle(v_1,\dots,v_n,0),(0,\dots,0,1)\rangle \subset \bbR^{n+1}.\]
 So, if $\co(\mathcal{S}_G)$ is not closed, then $\co(\mathcal{S}_K)$
 is not closed either, and this is enough to show that neither
$\bbC[\bfx,\bfy]^G$ nor $\bbC[\bfx,\bfy]^K$ are finitely generated.

We want to show that the convex cone $\co(\mathcal{S}_G)$
spanned by the support semigroup
\[
\mathcal{S}_G=\{(d,m) \, | \, [I(mZ_\bfv)]_d \ne 0\} \subset \bbZ^2.
\]
is not closed for suitable $\bfv$.
Set $\delta=\sqrt {\sum_{j=1}^n v_j^2}$.
The  symbolic powers $I(mZ_\bfv)$ form a \emph{multiplicative filtration},
i.e.,
\begin{equation}\label{filtr}
I(mZ_\bfv)I(m'Z_\bfv)  \subseteq I((m+m') Z_\bfv)
\end{equation}
  in particular
$(I(mZ_\bfv))^{\ell} \subseteq I(\ell m Z_\bfv)$.

\begin{example}
If $p_1,\ldots,p_{10}\in\bbP^2$ are the 10 nodes of
an irreducible nodal rational sextic, then for
 $Z=p_1+\cdots+ p_{10}$, one has $I(Z)_3= 0$, hence $I(Z)_3 I(Z)_3= 0$; but $I(2Z)_6\neq 0$, thus $(I(Z))^{2} \subsetneq I(2 Z)$.
\end{example}
For any homogeneous ideal $I$  in $\bbC[\bbP^2]$,
let $\alpha(I)=\min\{t\,|\,I_t\ne 0\}$.
\begin{lemma}\label{nagata_lemma1}
 Suppose that for every $m\ge 1$  it is $\alpha(I(mZ_\bfv))>m\delta$. Then
 for every  $m\ge 1$ there is $\ell >1$ such that
 $(I(mZ_\bfv))^\ell \subsetneq I({\ell m}Z_\bfv)$.
\end{lemma}
\begin{proof}
By \eqref{filtr}, $\alpha(I(mZ_\bfv))$ is a subadditive sequence,
hence, by the Fekete Lemma, the limit $\underset{m\to \infty}\lim \frac
{\alpha(I(mZ_\bfv))}m$ exists, and it equals
\[\hat{\alpha}(I(Z_\bfv))=\inf
\left\{\left.\frac {\alpha(I(mZ_\bfv))}m \right| m>0 \right\} ,\]
which is called the \emph{Waldschmidt constant} of $I(Z_\bfv)$.
Since
\[ \dim [I(Z_\bfv)]_d \ge\frac {d^2-m^2 \delta^2} 2 +\cdots, \]
where  the dots denote lower degree terms (see Harbourne's notes \cite{Har16}), we have
$\underset{m\to \infty}\lim \frac {\alpha(I(mZ))}m \le \delta$.
On one hand, by hypothesis
$\frac {\alpha(I(mZ))}m > \delta$ for all positive integers $m$. Hence,
\[
\lim_{\ell\to \infty} \frac {\alpha(I(\ell mZ))}{\ell m}
=
\lim_{m\to \infty} \frac {\alpha(I(mZ))}m = \delta.
\]
On the other hand, $\alpha\left((I(mZ))^\ell\right)=\ell\alpha(I(mZ))$
for every $\ell$,
\[\frac {\alpha\left(\left(I(mZ)\right)^\ell\right)}{\ell m}=
\frac {\alpha(I(mZ))}{m} >\delta\]
from which we conclude that for some large $\ell$ (depending on $m$)
$\alpha((I(mZ))^\ell)>\alpha(I(\ell mZ))$ and the claim follows.
\end{proof}

\begin{exercise}\label{waldschmidt}
 Let $Z_\bfv=v_1p_1+\dots+v_np_n$ be a nonzero fat point subscheme of $\bbP^2$.
Show that $1\le\widehat\alpha(I(Z_\bfv))\le \delta$.
\\
 \emph{Hint}:
look at $[I(kmZ)]_{kd}$
where $d/m$ is rational and close to but bigger than
$\delta$ and $k \gg 0$.
(See also  Exercise 1.3.6 in Harbourne's notes \cite{Har16}).
\end{exercise}

 \subsection{Ciliberto-Miranda's proof for Nagata's theorem}
 \subsubsection*{The Severi variety and degenerations}
By assigning the multiplicities $\bfv$ to any choice of $n$ points of
$\bbP^2$, one gets a scheme
$Z_\bfv=v_1p_1+\dots+v_np_n$ as above. The ideal
$I(Z)$ of course depends on the choice of the points.
Nagata's theorem deals with $\bfv=(1,\dots,1)$ and
a square number of \emph{very general} points, i.e.,
outside of a countable union of proper closed subsets of
$(\bbP^2)^n$.

We will follow the usual convention that, when a claim is made for
\emph{general} points, it is meant that that claim is
satisfied for every choice of the points outside
a proper closed subset of $(\bbP^2)^n$. Similarly when
dealing with a collection of objects (e.g., valuations)
parameterized by some variety $X$, claiming a fact for
general (resp. very general) objects will mean that all
objects parameterized by a Zariski open subset of $X$
(resp. a countable intersection of Zariski opens)
satisfy the claim.

\begin{theorem}[Nagata \cite{Nag58}]\label{Nagata-generic}
Let $\delta\ge4$ be an integer.
  If $p_1, \dots, p_{\delta^2}$ are very general points in $\bbP^2$,
and $Z=p_1+\dots+p_{\delta^2}$, then $\alpha(I(mZ))>\delta m$ for all $m\ge 1$.
\end{theorem}

Assigning a point $p$ of multiplicity $m$ to a homogeneous polynomial of
fixed degree $d$ corresponds to a set of $\binom{m+1}2$ linear
equations on the coefficients of the polynomial. As the position of the
assigned point varies, the coefficients determining these linear equations vary
polynomially in the coordinates of the point.
Thus, for each $d$ and $\bfv=(v_1,\dots,v_n)$ there are $\sum\binom{v_i+1}2$
equations determining a (possibly empty) ``Severi variety''
\[V_{\bfv,d} \subset (\bbP^2)^n \times \bbP(\bbC[w_1,w_2,w_3]_d)\]
formed by the closure of the set of the tuples $(p_1,\dots,p_n,F)$ such that $F$ has multiplicity
at least $v_i$ at $p_i$, i.e., the fibres of $V_{\bfv,d}$ for the
projection to $(\bbP^2)^n$ are the (projectivized)
degree $d$ pieces of the ideals $I(Z_\bfv)$ as the points
in $Z_\bfv$ vary.

Since the Severi variety is Zariski-closed
and the projection to $(\bbP^2)^n$ is a projective map,
general fibers of the Severi variety $V_{\bfv,d}$
are nonempty exactly when
the image of $V_{\bfv,d}$ is the whole $(\bbP^2)^n$.
Moreover, if we set  $\alpha_{\rm{gen}}(\bfv)$ the value of
$\alpha(I(Z_\bfv))$ for general $p_i$,
then for every $0<d<\alpha_{\rm{gen}}(\bfv)$ the image of $V_{\bfv,d}$
on $(\bbP^2)^n$ is a closed proper subset, and therefore
$\alpha(I(Z_\bfv))=\alpha_{\rm{gen}}(\bfv)$ for all choices of points
$p_i$ off these (finitely many) closed subsets.
This allows for \emph{specialization} and \emph{degeneration}
arguments: if there is some position of the points such that $[I(Z_\bfv)]_d=0$,
then the same holds for general points and so
$\alpha_{\rm{gen}}(\bfv)>d$. Thus Theorem
\ref{Nagata-generic} is equivalent to:

\begin{theorem}\label{Nagata-general}
  Let $\delta\ge4$, $m\ge 1$ and $d\ge 1$ be integers with $d\le \delta m$.
  If $p_1, \dots, p_{\delta^2}$ are  general points in $\bbP^2$,
and $Z=p_1+\dots+p_{\delta^2}$, then $[I(mZ)]_d=0$.
\end{theorem}

A semicontinuity argument was used by Nagata
to prove his theorem, and this is also the route
we shall follow here, adapting a plane curves degeneration  argument
of Ciliberto and Miranda \cite{CM06}, to prove it.

Consider $\pi: Y \to \mathbb D$ the family
obtained by blowing up  the trivial family $\mathbb D\times {\mathbb{P}}^2\to \mathbb D$ over a disc $\mathbb D$
at a point in the central fiber.
The general fibre $Y_u$ for $u\neq 0$ is a ${\mathbb{P}}^2$,
and the central fibre $Y_0$ is the union of two surfaces
$\bbP \cup \bbF$, where $\bbP \cong {\mathbb{P}}^2$ is
the exceptional divisor and
$\bbF \cong {\mathbb{F}}_ {1}$ is the original central
fibre blown up at a point.
The surfaces $\bbP$ and $\bbF$ meet transversally along a rational curve $E$
that is the negative section on $\bbF$ and a line on $\bbP$.

One can split $n$ as a sum $n=a+b-1$, and choose
$a$ points $q_1, q_2, \dots, q_a \in \bbP\setminus E$,
and $b-1$ points $q_{a+1}, q_{a+2}, \dots, q_{n} \in\bbF\setminus E$.
Consider these $n$ points as limits of $n$
general points in the general fibre $Y_u$, i.e., fix $n$
sections $\sigma_1,\dots,\sigma_n$ of $Y \to \mathbb D$ going through the chosen points.
These sections determine a map
$\mathbb D \setminus \{0\} \rightarrow (\bbP^2)^n$. Consider the scheme
$Z_\bfv=v_1p_1+\dots+v_np_n$,
if $[I(Z_\bfv)]_d$ is nonempty for a general choice of points,
then (pulling back from the Severi variety) there is a family
of curves $C\subset (\mathbb D \setminus \{0\})\times \bbP^2$
of degree $d$ such that the fiber $C_u$ over every
$u\ne0$ has multiplicity at least $v_i$ at the point
$\sigma_i(u)\in \bbP^2_u$. The closure $\bar C \subset Y$
of $C$ in $Y$ has a ``central fiber'' $C_0$ which is the union
of a curve in each component of $Y_0$, 
$C_0=C_{\bbF}+C_{\bbP}$, and has multiplicity 
at least $v_i$ at each $p_i$
(because $p_i$ is a smooth
point of $Y_0$ and of $Y$, so that the section
$\sigma_i$ meets $Y_0$ transversely at $p_i$).
More explicitly,
$C_{\bbF}$ is the proper transform in $\bbF$ of a  curve  of degree $d$, with some multiplicity
$e$ at the blown up point and multiplicities $(v_{a+1},...,v_{n})$
at the $b-1$ points $Z_{\bbF}$ in $\bbF$, and $C_{\bbP}$ is the proper transform  in
$\bbP$ of a curve of degree $e$ and multiplicities $(v_{1},...,v_{a})$
at the other chosen points $Z_{\bbP}$. Moreover, the two curves have the
same intersection with the rational curve $E$, that is
\begin{equation}\label{centralcurve}
C_{\bbF}\cap E=C_{\bbP}\cap E.
\end{equation}

In other words, we have a  family of curves $\bar C$ that fits in the diagram
\begin{equation}\label{degen}
\begin{gathered}
\xymatrix{
\bar C  \ar[d] \ar[r] & Y = \mathrm{Bl}(\mathbb D\times {\mathbb{P}}^2) \ar[d]^\pi
\\
\mathbb D \ar[r]^{=} &  \mathbb D
}
\end{gathered}
 \end{equation}
in which the specialized curve over $0\in \mathbb D$ splits with the splitting of the surface $Y_0=\bbP \cup \bbF$ in the central fiber of the family of surfaces. The scheme of points contained in the general curve also splits with the curve in the central fiber of $\bar C$.

The preceding discussion can be summarized by saying
that the limit of a family of Cartier divisors is a union
of divisors matching their intersections on $E$.
We refer to \cite{CM98}, \cite{CM06} and \cite{CHMR13}
for more on these particular degenerations.

Theorem \ref{Nagata-general} will follow from the following two lemmas, which will be proved in the next subsection.

\begin{lemma}\label{CM-induction}
  Let $\delta\ge4$ and $m\ge 1$ be integers and assume that
  for $p_1, \dots, p_{\delta^2}$ general points in $\bbP^2_\bbC$,
and $Z=p_1+\dots+p_{\delta^2}$, one has $\alpha(I(mZ))>\delta m$.
Then if $p_1, \dots, p_{(\delta+1)^2}$ are general points in $\bbP^2_\bbC$,
and $Z'=p_1+\dots+p_{(\delta+1)^2}$,  one has $\alpha(I(mZ'))>(\delta+1) m$.
\end{lemma}

\begin{lemma}\label{CM-base}
    Let $m\ge 1$ be an integer.
  If $p_1, \dots, p_{16}$ are general points in $\bbP^2_\bbC$,
and $Z=p_1+\dots+p_{16}$, then $\alpha(I(mZ))>4 m$.
\end{lemma}

\subsubsection*{Cremona maps}

For the proof of Lemma \ref{CM-induction}, it will be useful
to exploit some particular Cremona transformations.
Computations are very explicit
and will be left as exercises.
The general theory of Cremona maps (birational maps of $\bbP^2$)
including the description of their effect on plane curves,
can be found in \cite{Alb02}.

Recall that, given three points $p,q,r$ in $\bbP^2$, not on a line,
$\dim [I(p+q+r)]_2=3$, and three independent quadratic forms vanishing
at $p, q, r$ define a birational map $\bbP^2 \dashrightarrow \bbP^2$,
called a \emph{standard Cremona map}.
This map is defined everywhere except at $p, q, r$ and contracts the line
$p\wedge q$ to a point $r'$, the line $q\wedge r$ to $p'$ and the line $p\wedge r$ to $q'$.
The standard Cremona map based at $p', q', r'$ is the inverse of the
previous map, i.e., the composition of both maps is the identity
on the complement of the triangle determined by $p, q, r$.
\\
 We say that a collection of $n\ge 3$ points in $\bbP ^2$ is in
\emph{linear general position} if no subset of $3$ points
is contained in a line; in particular, the points are all distinct.
Note that, given $n$ points in linear general position,
we can perform the standard Cremona transformation on any
subset of 3 points among the $n$ points.
This gives a different collection of $n$ points in $\bbP^2$, that need not be in linear general position.
\\
We say that a collection of $n$
points is \emph{in Cremona general position} if they are in linear general position and
this remains true after any finite sequence of standard Cremona transformations
on subsets of 3 points.
\begin{exercise}[Transforming curves by standard Cremona maps]
\label{transform-cremona}
Taking projective coordinates $x,y,z$
with vertices at $p, q, r$, the Cremona map
based at $p,q,r$ is given by $(x:y:z)\mapsto(yz:xz:xy)$.
Check that the points $p', q', r'$ coincide with $p,q,r$
and this map is its own inverse. Therefore, direct image
and proper preimage of curves (i.e., disregarding
components supported on the coordinate triangle)
of curves under this standard Cremona map coincide.
\\
Show that a curve of degree $d$ with multiplicities
$m_p, m_q, m_r$ at the three
given points is mapped by the Cremona map to a curve of degree
$d+c$ with multiplicities $m_{p'}=m_p+c$, $m_{q'}=m_q+c$, $m_{r'}=m_r+c$
at the three distinguished points in the image, where
$c=d-m_p-m_q-m_r$.
 Any singularity off the
triangle with vertices $p,q,r$ is preserved because the Cremona map
acts as an isomorphism there.
\\
\emph{Hint}: Plugging the expression of the Cremona map into
the equation of the curve shows that the preimage curve
has degree $2d$; check that its equation contains the factor
$x$ exactly $m_p$ times, $y$ exactly $m_q$ times and $z$ exactly
$m_r$ times, to obtain the proper preimage.
\end{exercise}

\begin{exercise}[Openness conditions for collections of points in $(\bbP^2)^n$]
Show that, for every positive integer $\delta$, the locus in $(\bbP^2)^n$
of $n$-tuples of points that are in linear general position, and such that
this remains true after a sequence of $k\le \delta$
standard Cremona transformations on subsets of 3 points, is Zariski open.
\\
Show that the locus in $(\bbP^2)^n$
of $n$-tuples of points in Cremona general position is
the intersection of at most countably many
 Zariski-open subsets of $(\bbP^2)^n$.
\end{exercise}

\begin{remark}
By \cite[Theorem 5.7.3]{Alb02} (a result apparently first
stated by H.~P.~Hudson and proved by P.~Du~Val), the locus in $(\bbP^2)^n$
of $n$-tuples of points in Cremona general position is
Zariski-open if and only if $n\le 8$. See also
\cite[Example V.4.2.3 and Exercise V.4.15]{Har77}.
\end{remark}

\begin{exercise}
 Let $\delta$, $m$ and $e$ be positive integers with
 $e\ge\delta m$, and let $d=(\delta+1)m$, $\Delta=e-\delta m$ and $n=2\delta+1$.
 Pick points $p_1,\dots,p_n\in\bbP^2$ in general position.
 Show that a plane curve of
degree $d$, with multiplicity $e$ at $p_1$ and multiplicity $m$
at each of $p_2, \dots, p_r$ can be transformed by
a sequence of standard Cremona maps
into a curve
of degree $m-\Delta \delta$ with a point of multiplicity $m$.
\end{exercise}

\begin{proof}[Proof of Lemma \ref{CM-induction}]
We argue by contradiction. Assume that $\alpha(I(mZ'))\le(\delta+1) m$,
which means that $[I(mZ')]_{(\delta+1)m}\ne 0$, and consider
the degeneration \eqref{degen}, where the $(\delta+1)^2$ general points
in the general fiber will degenerate to $(\delta+1)^2$ points
in the special fiber, $a=\delta^2$ which can be assumed to be
general on the surface $\bbP$,
and $b-1=2\delta+1$ which can be assumed to be
general on the surface $\bbF$. Since
$[I(mZ')]_{(\delta+1)m}\ne 0$ for general $Z'$, we obtain a family
of curves and a central curve, as in \eqref{centralcurve},
consisting of a curve $C_\bbF$ of
degree $t=(\delta+1)m$, with some multiplicity
$e$ at the blown up point and multiplicity $m$
at the $b-1=2\delta+1$ points chosen in $\bbF$, plus a curve
$C_\bbP$ of degree $e$ and multiplicity $m$
at the $a=\delta^2$ general points.
By hypothesis, if $e\le\delta m$ such a curve does not
exist in $\bbP$, so it will be enough to prove that the claimed curve
in $\bbF$ does not exist for $e>\delta m$.

Let $\Delta=e-\delta m$, it is positive by hypothesis.
It was seen in the preceding exercise that after a sequence of $\delta$
Cremona transformations centered at the three biggest
multiplicities, a curve of degree $(\delta+1)m$  with multiplicity
$e$ at a general point and multiplicity $m$
at further $t-1=2\delta+1$ points would give a curve
of degree $m-\Delta \delta$ with a point of multiplicity $m$.
Obviously this is impossible, as $\Delta>0$.
\end{proof}

\begin{proof}[Proof of Lemma \ref{CM-base}]
We try to apply the same degeneration argument to a general
curve of degree $4m$ with 16 assigned points. In this case
 the only output is that a possible central curve
 \eqref{centralcurve} would consist of a curve
$C_\bbF$ of degree $4m$, with multiplicity
$e=3m$ at the blown up point and multiplicity $m$
at the $7$ points chosen in $\bbF$, plus a curve
$C_\bbP$ of degree $3m$ and multiplicity $m$
at the $9$ general points. Such curves do exist. In this
case the key point is that they cannot match on $E$ for
general points.

Indeed, $C_\bbP$ can only be the unique cubic through
the 9 general points taken $m$ times, whereas
$C_\bbF$ consists of $m$ curves in the pencil
of quartics with a triple point and 7 simple points
(this follows from the Cremona transformations
as in the previous lemma). These can only match on $E$
if the curve in $\bbF$ consists of $m$ times one single curve  in the pencil
of quartics, that
matches the cubic of $\bbP$, i.e., it meets $E$ at the
same three points. Now, the pencil of quartics induces
a pencil of degree 3 on $E$, i.e. a non-complete linear
series of degree 3. Choose 3 points on $E$ that do not belong
to this pencil, choose a cubic $C\subset \bbP$ through these 3 points,
and choose the 9 points on $\bbP$ as general points of $C$.
Then the matching is not possible; therefore it is not possible
for general points either.
\end{proof}
\subsection{Generalization to an arbitrary number of points}
The Ciliberto-Miranda method works more generally to yield the following.

\begin{theorem}[Ciliberto-Harbourne-Miranda-Ro\'e \cite{CHMR13}]\label{thm:chmr}
 For every $n\ge 10$ there exist multiplicities $\bfv=(v_1,\dots,v_n)$
 such that, if $p_1, \dots, p_n \in \bbP^2$ are very general points,
 $Z_\bfv=v_1p_1 + \dots + v_n p_n$ and $\delta=\sqrt{\sum v_i^2}$,
 then $\alpha(I(mZ_\bfv))>\delta m$ for all $m>1$. In particular,
 $\widehat\alpha(I(Z_\bfv))=\delta$.
\end{theorem}

The method of proof is essentially the same as for Nagata's theorem:
there are three initial cases $n=10,11,12$ and an induction step.
The $n=10$ case is slightly more difficult, and we refer the reader to \cite{CHMR13}
for the complete proof, that uses the same basic principle with
  a modified degeneration obtained by blowing up the central fiber
  of $Y\rightarrow \mathbb D$ along a suitable rational curve.
 The vector of multiplicities in this case is
 $\bfv=(5,4^{9})$ (so that $\delta=13$), and for very general points
  $p_1,\dots,p_{10}\in \bbP^2$,
 and $Z_\bfv=5p_1+4(p_2+\dots+p_{10})$, the inequality
  $\alpha(I(mZ_\bfv))>\delta m$ holds for all $m\ge 1$. The initial cases
  $n=11,12$ and the induction step are left as exercises:

\begin{exercise}$(n=12)$
 Prove that for $\bfv=(2^{8},1^4)$ (so that $\delta=6$),
 and $Z_\bfv=2(p_1+\dots+p_8)+p_9+\dots+p_{12}$
 where $p_1,\dots,p_{12}\in \bbP^2$ are very general points,
  $\alpha(I(mZ_\bfv))>\delta m$ for all $m\ge 1$.
  \\
  \emph{Hint}:
use the Ciliberto-Miranda method of the first section, with
all four $m$-fold points on $\bbP$ and all eight
$2m$-fold points on $\bbF$.
\end{exercise}

\begin{exercise}$(n=11)$
 Prove that for $\bfv=(3,2^{10})$ (so that $\delta=7$),
 and $Z=3p_1+2(p_2+\dots+p_{11})$ where
 $p_1,\dots,p_{11}\in \bbP^2$ are very general points,
  $\alpha(I(mZ_\bfv))>\delta m$ for all $m\ge 1$.
  \\
  \emph{Hint}:
use the Ciliberto-Miranda method of the first section,
with four of the $2m$-fold points on $\bbP$ and the
rest on $\bbF$.
\end{exercise}

\begin{exercise}  \label{CHMR-induction}
(induction step)
Let $a,b,d$ be positive integers, and assume the multiplicity
vectors $\bfm=(m_1,\dots,m_a)$, $\bfn=(n_1,\dots,n_b)$
and $\mu=(\mu_1, \dots, \mu_b)$
satisfy:
\begin{enumerate}
 \item For $p_1, \dots, p_{a}$ very general points in $\bbP^2$,
 and letting $Z_\bfm=m_1p_1+\dots+m_ap_a$, then for each $m\ge 1$, one has
 $\alpha(I(mZ_\bfm))>d m$.
 \item For $p_1, \dots, p_{b}$ very general points in $\bbP^2$,
 and letting $Z_\mu=\mu_1p_1+\dots+\mu_bp_b$, one has
 $\widehat\alpha(I(Z_\mu))=\delta=\sqrt{\sum \mu_i^2}$.
 \item $\sum \mu_i n_i \ge \delta m_1$.
\end{enumerate}
Show that in this case the multiple point scheme
$Z_{\bfm\sharp\bfn}$ determined by the
multiplicity vector $\bfm\sharp\bfn=(n_1,\dots,n_b,m_2,\dots,m_a)$
at $n=a+b-1$ very general points
satisfies that for each $m\ge 1$, it is $\alpha(I(mZ_{\bfm\sharp\bfn}))>d m$.
\\
\emph{Hint}: Show, using the ideas of Exercise \ref{waldschmidt},
that if $\alpha(I_{Z_\bfn})<m_1$, with $\sum \mu_i n_i \ge \delta m_1$,
then $\widehat{\alpha}(I_{Z_\mu})<\delta$.
See also Exercise \ref{ex:extremalQ}.
\end{exercise}

\subsubsection*{How general need the points be?}
Nagata's theorem can be rephrased in terms of \emph{semi-effective divisors}
\cite[Section 1.2]{Har16}. Consider the blow-up $\pi:X\rightarrow \bbP^2$ of $\bbP^2$ at the points
$p_1,\dots,p_n$, denote by
$L$ the pull-back to $X$ of the class of a line, by $E_i$ the class of
the exceptional divisor above $p_i$. A divisor class $D$ is called semi-effective
if for some positive integer $m$ one has $H^0(X,\cO_X(mD))\ne 0$,
i.e., if for some $m > 0$, the divisor $mD$ is linearly equivalent to an effective divisor.
Nagata's theorem is equivalent to the fact that, if $n=\delta^2$ and
$p_1,\dots,p_n$ are very general points, then
the divisor $D_{\delta}=\delta L-E_1-\dots-E_n$
is not semi-effective.

Clearly, without some generality assumption, the divisor $D_\delta$
can be semi-effective (in fact, it can be effective: it suffices to choose
$\delta^2$ points on a curve of degree $\delta$; this shows
that the Severi variety $V_{1,\delta}$ is nonempty).
Harbourne's notes \cite{Har16} raise the question of how large  the
least integer $m$ such that $H^0(X,\cO_X(mD))\ne 0$ can be, when a
divisor $D$ is semi-effective; let us consider the case $D=D_\delta$.
Looking at the Severi varieties $V_{m,\delta m}$ corresponding to multiples of $D_\delta$,
 one easily sees that
\[V_{1,\delta } \subset V_{m,\delta m}\subset V_{km,k\delta m}\]
for all $m$ and $k$ (in particular they are nonempty). On the other hand, the
naïve expectation (which is a lower bound) for the dimension of the
Severi varieties obtained by counting equations is
\[\dim V_{m,\delta m}\ge 2n + \binom{\delta m+2}{2}-
n\binom{m+1}{2} =\frac{(3-\delta)\delta}{2} m +2 \delta^2+1, \]
that is a strictly decreasing function of $m$ if $\delta\ge 4$.
So one naïvely would \emph{not} expect $V_{km,k\delta m}$ to
be strictly larger than $V_{m,\delta m}$ for large $m$; in other words,
it is conceivable that the answer to the following problem is positive:

\begin{problem}
 Is there any bound $m_0=m_0(\delta)$ such that if the
points $p_i$ are chosen so that $mD_\delta$ is not effective for
all $m\le m_0$, then $D$ is not semi-effective?
\end{problem}

If the answer to this question is positive, then the set of $n$-tuples
of points for which Nagata's theorem is true would be a Zariski open
set. One among many consequences that would follow is, for example,
that there would be sets of points with coordinates in $\bbQ$
(or any infinite field) with the Nagata property. Note that Totaro's work
\cite{Tot08} shows that there exist sets of 9 points in $\bbP^2_\bbQ$
with non finitely generated multigraded Rees algebra, but the support
in this case spans a closed cone.

 \subsection{Nagata's conjecture}

 \begin{conjecture}[Nagata, 1959]
   Let $n\ge10$ be an integer.
  If $p_1, \dots, p_{n}$ are generic points in $\bbP^2$,
and $Z=p_1+\dots+p_{n}$, then $\alpha(I(mZ))>m\sqrt{n}$ for all $m\ge 1$.
 \end{conjecture}

This statement holds true for $n$ a square, by Nagata's theorem,
but it remains open for all other values of $n$.

\begin{remark}
  For nonsquare $n$ the equality
$\alpha(I(mZ))=m\sqrt{n}$ is impossible. Therefore, to prove
Nagata's conjecture it is enough to compute the Waldschmidt constants:
$\widehat\alpha(I(Z))=\sqrt{n}$ for all $n\ge 10$.
\end{remark}

\begin{remark}
 Some bounds are known for $\widehat\alpha(I(Z))$ that approximate
 the square root of the number of points. For instance, \cite{HR09} gives
 \[ \sqrt{n} \,\le\, \widehat\alpha(I(Z)) \,\le\,
 \sqrt{n}\,\sqrt{1+\frac{2}{n^2-5n\sqrt{n}-2}},\]
 if $Z=p_1+\dots+p_n$ consists of general points. Observe that the
 upper bound is a worst case estimate, as all known methods for
 obtaining such bounds give in fact rational numbers.
\end{remark}

One can also look at the question in terms of Seshadri constants
\cite{Dem92}.
The \emph{Seshadri constant} of a set of points $p_1, \dots, p_n$ is
defined as
\[ \varepsilon(p_1, \dots, p_n)=\inf
\left\{\frac{\deg C}{\sum \mult_{p_i}C} \right\},\]
where the infimum is taken with respect to all plane curves
passing through at least one of the points $p_i$.
Equivalently, and denoting as before $\pi:X\rightarrow \bbP^2$
the blow-up at the $n$ points,
$L$ the pull-back to $X$ of the class of a line, $E_i$ the class of
the exceptional divisor above $p_i$,
\[ \varepsilon(p_1, \dots, p_n)=\sup
\left\{ t \in\bbR \left | L- t (E_1+\dots+E_n) \text{ is nef}\right.\right\}
\]

\begin{exercise}
For all choices of $p_1, \dots, p_n \in \bbP^2$, prove that
$ \varepsilon(p_1, \dots, p_n)\le 1/\sqrt{n}$.
\\
\emph{Hint}: Use Exercise \ref{waldschmidt}.
\end{exercise}

In fact, we will see below that
$\varepsilon(p_1,\dots,p_n)=\widehat\alpha(I(Z))^{-1}$ for
$Z=p_1+\dots+p_n$. Therefore, Nagata's conjecture is equivalent
to the claim that, for very general points $p_1,\dots,p_n\in \bbP^2$,
\[ \varepsilon(p_1,\dots,p_n)=1/\sqrt{n}.\]
Following this approach, one can formulate an analogous conjecture
for arbitrary surfaces:

\begin{conjecture}[Biran--Szemberg \cite{Sze04}, {\cite[Remark 5.1.24]{Laz04I}}]\label{SLC}
Let $X$ be a smooth projective surface and $L$ be a nef
  divisor on $X$.  Then there is a
  positive integer $n_0$ such that for every $n\ge n_0$, one has
$\epsilon(n;X,L) = \sqrt{L^2/n}.$
\end{conjecture}

\begin{remark}
If there exists a smooth curve of positive genus
in the linear system $|kL|$ then $n_0=k^2L^2$
is expected to work in Conjecture \ref{SLC}.
\end{remark}

\bigskip

Summarizing, Nagata's approach to showing that the $k$-algebra $A$ is not
finitely generated has three steps:
\begin{enumerate}
 \item $A$ is isomorphic to a bigraded ring $\oplus_{m,d} A_{m,d}$
 (in fact with  $A_{m,d}= [I(mZ)]_d$),
 \item the \emph{support} $\mathcal{S}=\{(m,d)\,|\, A_{m,d} \ne 0\}$
 is not finitely generated as a semigroup, because
 \item the cone $\co(\mathcal S) \subset \bbR^2$ is not closed.
\end{enumerate}
In these notes, we call \emph{Nagata-type statement} a theorem or
conjecture stating that, for a given multigraded algebra $A$
with some geometric meaning (most often a Rees algebra)
the cone $C$ in $\bbR^n$ spanned by the support semigroup of $A$
is not closed. Of course, this implies that $A$ is not finitely generated.
In the same spirit, we say that a ray $R=co(v)\subset \bbR^n$
is a Nagata-type ray if $R\subset \overline C \setminus C$.

\begin{exercise}
 Show that, if $C$ is a closed cone in $\bbR^2$, then it is finitely generated.
Give an example of a closed cone in $\bbR^3$ that is not finitely generated.
Is it true that if $S\subset \bbZ^2$ spans a closed cone, then $S$ itself is
finitely generated?
\end{exercise}

\section{Conjectures on the cone of curves}

In this section we review some conjectures that generalize
and strengthen Nagata's conjecture, following ideas from \cite{CHMR13}.
We fix for most of the section the hypothesis that $p_1,\dots,p_n$ are $n$
very general points of $\bbP^2$ and consider the blow-up
$f: X=X_n\to \bbP^2$ of the plane at the points $p_1,\ldots, p_n$.
The object of interest is the geometry of $X$,
from the point of view of Mori theory.
More precisely, we shall see what Nagata's
conjecture and its generalizations tell about the shape of
the Mori cone of $X$.

\subsection{Total coordinate ring of the blown up plane}

The \emph{Picard group} $\Pic (X)$
of the blown-up plane $f: X\to \bbP^2$
is the abelian group freely generated by:
\begin {itemize}
\item the \emph{line class}, i.e., the pullback $L=f^*(\cO_{\bbP^2}(1))$;
\item the classes of the  \emph{exceptional divisors} $E_1,\ldots, E_n$ that are contracted to $p_1,\ldots, p_n$.
\end{itemize}
So, \(
\Pic X\cong \mathbb{Z}L\oplus\mathbb{Z} E_1\oplus\cdots\oplus\mathbb{Z}E_n\cong \mathbb{Z}^{n+1}.
\)
More generally we consider $\bbQ$ and $\bbR$-divisor classes and work in
$N^1(X)=\NS(X)\otimes _\bbZ\bbR=\Pic(X)\otimes _\bbZ\bbR\cong\bbR^{n+1}$,
viewed as a real vector space with its standard Euclidean topology.
The real cones spanned by effective (or ample, or nef, etc) divisors
are objects of great interest to understand the geometry of $X$.
This approach was pioneered
by Kleiman in \cite{Kle66} and is explained in detail in \cite[1.4.C]{Laz04I}.

A class $\xi\in N^1(X)$ is \emph{integral} (respectively \emph{rational})  if it sits in $\Pic(X)$ (respectively in $\Pic(X)\otimes _\bbZ \bbQ$). A ray $\co(\xi)$
in $N^1(X)$ is \emph{rational} if it is generated by a rational class.
A rational ray in $N^1(X)$ is \emph{effective} if it is generated by an effective  class.

We will use the notation $\cL=(d;m_{1},\ldots, m_{n})$  for the
complete linear system
\[\cL=\left| dL-\sum_{i=1}^n m_iE_i\right|=
\bbP\left(H^0\left(X,\cO_X\left( dL-\sum m_iE_i\right)\right) \right)\]
on $X$. With this convention the integers $d,m_{1},\ldots, m_{n}$ are the
components with respect to the ordered  basis $(L, -E_1,\ldots, -E_n)$ of $N^1(X)$,
and the intersection form on $N^1(X)$ can we written as follows:
\[ (d; m_1, \dots, m_n)\cdot(d';m'_1,\dots,m'_n)=dd'-m_1m'_1-\dots-m_nm'_n.\]
We also use the shorthand exponent notation, so that $m^k$
denotes $k$-fold repetition of the integer $m$.
Thus the \emph{canonical class} on $X$ is $K_n=(-3; -1^n)$
(called $K$ if there is no danger of confusion).

The pull-back by the blow-up map $f$
induces a natural isomorphism for each $(d, \bfm)\in \bbZ^{n+1}$
\begin{equation}\label{coxrees}
 H^0\left(X,\cO_X\left( dL-\sum m_iE_i\right)\right)
\,\cong\, \left[I(Z_\bfm)\right]_d,
\end{equation}
where $Z_\bfm=m_1p_1+\dots+m_np_n$ as usual.
Therefore, there is an isomorphism between the
\emph{total coordinate ring} (also called Cox ring) of $X$
\[
\mathcal{T C}(X)=\underset{L\in\Pic X}{\mathsmaller{\bigoplus}} H^0(X;L)
\]
and the multigraded Rees algebra
\[
\underset{\bfm \in \bbZ^n, d\ge 0}{\mathsmaller{\bigoplus}} [I(Z_\bfm)]_d,
\]
that by Lemma \ref{fixedring} can be identified with the ring of
invariants of the unipotent action.
The isomorphisms \eqref{coxrees} are compatible with
product operations on each side, so
$\mathcal{T C}(X)$ and $\oplus_{\bfm,d} [I(Z_\bfm)]_d$
are isomorphic as graded algebras.

The support semigroup of $\mathcal{TC}(X)$ is by definition $\Eff X$,
the semi-group in $\Pic X$ of effective classes on $X$, that is,
\[
\mathcal{S}_K=Supp\left(
\underset{\bfm \in \bbZ^n,d\in \bbZ}{\mathsmaller{\bigoplus}}
I(Z_\bfm)\right)\cong Supp(\mathcal{T C}(X) )=
\Eff X=\left\{L\in \Pic X | H^0(X;L)\neq 0\right\}.
 \]
As usual, the algebra $\mathcal{T C}(X)$ is not finitely generated if
its support $\Eff X$ is not so as semi-group.

The \emph{Mori cone} $\Mor(X)$ is the topological closure in
$N^1(X)\cong\bbR^{n+1}$  of the cone
\[\NE(X)=\co(\Eff X)\] of all effective rays, and it is the dual of the
\emph{nef cone} $\Nef(X)$, that is the closed cone
described by all nef rays.
We have
 \[\co(\mathcal{S}_K)\cong\NE(X)=\co(\Eff X) \subset N_1(X). \]
In this language, Nagata's theorem and the generalizations
seen in the first section provide a non-closedness result for $\NE(X)$:

\begin{corollary}\label{nonclosed-ne}
 For $n\ge 10$, there exists a rational ray $\co(\xi)$
 in $\Mor(X)$ that is not contained in
 $\NE(X)$.
\end{corollary}
\begin{proof}
 By Theorem \ref{thm:chmr},  for every $n\ge 10$ there exist multiplicities
 $\bfv=(v_1,\dots,v_n)$
such that, if
$\delta=\sqrt{\sum v_i^2}$,
 then $\alpha(I(mZ_\bfv))>\delta m$ for all $m>1$ and
 $\widehat\alpha(I(Z_\bfv))=\delta$.
 By the isomorphisms \eqref{coxrees}, this implies that
the ray spanned by $\xi_\bfv=(\delta;v_1,\dots,v_n)$
lies in $\Mor(X)\setminus\NE(X)$.
\end{proof}

 On the other hand, Nagata's conjecture is equivalent to
 the following:

\begin{conjecture}
For $n\ge 10$, the ray spanned by the class
$(\sqrt{n};1^n)$ is not contained in $\NE(X)$.
\end{conjecture}

We call $\nu_n=\co(\sqrt{n};1^n)$ the \emph{Nagata ray}.

\begin{remark}
Corollary \ref{nonclosed-ne} obviously implies that
$\Eff X$ is not finitely generated, but it is much stronger.
We will see in the next section that, for $n=9$, the cone
$\NE(X)=\Mor(X)$ is closed, but still $\Eff X$ is not finitely
generated.
\end{remark}

 Corollary \ref{nonclosed-ne} also implies that the ray spanned by
 $\xi_\bfv=(\delta;v_1,\dots,v_n)$ is extremal in
$\Mor(X)$. The existence of extremal rays of selfintersection
zero on every $X_n$, $n\ge 9$ was first proved by F.~Monserrat
in \cite{Mon11}.
The following set of exercises leads to a proof
of extremality of $\xi_\bfv$.

Let
\begin{equation}\label{nonegcone}
 \cQ=\cQ_n = \left\{ \xi\in N^1(X) \, \text{ such that }\,  \xi\cdot L\ge 0 \text{ and }  \xi^2\ge 0 \right\}\subset N^1(X)
\end{equation}
 be the
 \emph{nonnegative cone}. Clearly
the ray $\co(\xi_\bfv)$ lies on the boundary $\partial\cQ$
of $\cQ$.

\begin{exercise}
 Show that $\cQ\subseteq \Mor(X)$.
 \\
 \emph{Hint}: use Exercise \ref{waldschmidt}, applied to classes
 $\xi=(d;m_1,\dots,m_n)$ with $\xi^2>0$ via \eqref{coxrees}.
\end{exercise}
\begin{exercise}\label{negative-extremal}
Show that if $C$ is an irreducible curve on a surface $X$
with $C^2<0$, then $[C]$ belongs to every system of generators
of $\Eff X$ and the ray $\co([C])$ is extremal in $\Mor(X)$.
\\
\emph{Hint}: For every positive $d$, the unique effective
divisor in the complete linear system $|dC|$ is $dC$.
\end{exercise}
\begin{exercise}\label{ex:extremalQ}
Show that if a rational class $\xi \in \partial\cQ$
and a class $\eta\in \NE(X)\setminus \cQ_n$ satisfy
$\xi\cdot\eta<0$ then $\xi\in \NE(X)$.
\\
\emph{Hint}: Show that there exists an irreducible curve
$C$ with $C^2<0$ and $C\cdot \xi<0$. Deduce that for suitable
integers $d,\, m$, the class $d\xi-m[C]$ lies in the interior of $\cQ_n$
and hence in $\NE(X)$. Compare with Exercise \ref{CHMR-induction}
\end{exercise}

\begin{exercise}\label{exercise-extremal}
Show that the class $\xi_\bfv=(\delta;v_1,\dots,v_n)$
of corollary \ref{nonclosed-ne} is nef. Deduce that $\xi_\bfv$
is extremal in $\Mor(X)$.
 \\
 \emph{Hint}: If $\xi_\bfv$ were not extremal, then it could
 be written as a finite sum $\sum a_i[C_i]$ with $a_i$ nonnegative
 rational numbers and $C_i$ irreducible curves
 with $C_i^2<0$ and  $C_i\cdot \xi_\bfv=0$.
\end{exercise}
\subsection{Mori's cone theorem and consequences}

Let $K=K_n=(-3,-1^n)$ be the canonical divisor on $X$. For any subset
$S \subset N^1(X)$, let $S^\succcurlyeq$
(respectively $S^\preccurlyeq$, $S^\succ$ and $S^\prec$) be
the subset of $S$ consisting of the
 nonzero classes $\xi$ such that $\xi\cdot K_n\ge 0$
(respectively, $\xi\cdot K_n\le 0$, $\xi\cdot K_n> 0$
and $\xi\cdot K_n< 0$). Rays in $N^1(X)^\preccurlyeq$ spanned by
rational curves
play a special role in Mori's theory
(see theorem \ref{mori} below). The cone
spanned by them is called
\begin{equation*}
R_n\;= \co\left(\left\{[E]\,|\,{E\;\text {rational
		smooth curve with } 0\leq -E\cdot K_n}\right\} \right) \subseteq\; \Mor(X_n)^\preccurlyeq,
\end{equation*}
or simply $R=R_n$ if there is no danger of confusion.
A particularly important case is that of 
\emph{$(-1)$-rays} in $N^1(X)$, namely those spanned by the
class of a $(-1)$-\emph{curve}, i.e., a smooth, irreducible,
rational curve $E$ with $E^2=-1$ (hence $E\cdot K_n=-1$ by adjunction and so $[E]\subset R_n$). 
Every $(-1)$-ray $\co([E])$ is effective and extremal
in $\Mor(X)$ (Exercise \ref{negative-extremal}). 
On $X_n$, $n\ge2$ the cone $R_n$ is in fact spanned by
$(-1)$-rays:

\begin{equation}\label{Raggin}
R_n\;= \co\left(\left\{[E]\,|\,{E\;\text {a}\; (-1)\text{-curve}}\right\} \right) \subseteq\; \Mor(X_n)^\preccurlyeq.
\end{equation}
We now state  Mori's cone theorem in the form
 that is most useful for the rational surface $X_n$:

\begin{theorem}[Mori \cite{Mor82}, see {\cite[1.5F]{Laz04I}}]\label{mori}
 \[\Mor(X_n)= \Mor(X_n)^\succcurlyeq + R. \]
\end{theorem}

The \emph{negative part} $R$ in the Mori decomposition
of $\Mor(X)$ is well understood, whereas known descriptions of
$\Mor(X)^\succcurlyeq$ are only conjectural if $n\ge 10$.
We summarize a few facts known about $(-1)$-curves
that appropriately describe $R$.
 
Recall from Exercise \ref{transform-cremona} that applying
a standard Cremona map based at the points $p_1, p_2, p_3$
to curve of degree $d$
with multiplicity $m_i$ at the point $p_i$ transforms it into a curve of
degree $2d-m_1-m_2-m_3$ with multiplicity $d+m_i-m_1-m_2-m_3$
at the point $p_i$. The \emph{arithmetic Cremona transformation}
based at the points $p_i, p_j, p_k$ is the automorphism
\( N^1(X) \longrightarrow N^1(X)\)
that maps the class $(d;m_1,\dots,m_n)$ to
$(d';m_1',\dots,m_n')$, where
\begin{align*}
 d'&=2d-m_i-m_j-m_k, & m_\ell'&=m_\ell \; \forall \ell \notin \{i,j,k\} \\
 m_i'&=d-m_j-m_k, &  m_j'&=d-m_i-m_k, &
 m_k'&=d-m_i-m_j.
\end{align*}
A divisor $D$ is called 1--\emph{connected} (classically, \emph{virtually connected})
if it is effective and, for every decomposition
$D=D_1+D_2$ where $D_1$ and $D_2$ are effective,
$D_1\cdot D_2>0$ (cf. {\cite[Chapter II.12]{BPV}).
\begin{theorem}[Hudson-Nagata]\label{hudson-test}
Assume $n\ge 3$.
Let $\xi=(d;m_1,\dots,m_n) \in\Pic X$ be a class with $\xi^2=K\cdot \xi=-1$, with
$d\ge 0$ and $m_i\ge 0$.
The following are equivalent:
\begin{enumerate}
\item $\xi$ is the class of a $(-1)$-curve $E$.
\item $\xi$ is the class of a 1--connected divisor $D$.
\item  Recursively applying arithmetic Cremona transformations
based at the three points with largest multiplicities,
the degree $d$ decreases at each step and the final class
is a permutation of the multiplicities in $(0;-1,0^{n-1})$.
\end{enumerate}
\end{theorem}

We refer to \cite{Nag60} and \cite[Chapter 5]{Alb02} for proofs.
The third equivalent condition in theorem \ref{hudson-test}
is known as ``Hudson's test'' and can be effectively used to
find all $(-1)$ curves of a given degree in $X_n$.

\begin{exercise}
Show that, if $n\le 8$, there are finitely many $(-1)$-curves on
$X$ and finitely many Cremona maps whose indeterminacy locus
is contained in $\{p_1,\dots,p_n\}$. Specifically, justify the
numbers in this table (where $\#(-1)$ denotes the number of
$(-1)$-curves in $X_n$):
\begin{center}
\begin{tabular}{c|cccccc}
$n$ & 3 & 4 & 5 & 6 & 7 & 8 \\
\hline
\#$(-1)$& 6 & 10& 16 & 27 & 56 & 240\\
$\max \deg(-1)$ & 1 & 1 & 2 & 2 & 3 & 5 \\
\end{tabular}
\end{center}
Deduce that the locus of sets of $n\le 8$
points in Cremona-general position is a Zariski open set.
\\
\emph{Hint}: Start with the class $(0;-1,0^{n-1})$ and perform
all possible arithmetic Cremona transformations that increase
the degree. Compute the number of permutations of each
class obtained in this way.
\end{exercise}
It is not hard to see (use the Nakai-Moishezon Criterion,
 \cite[V.1.10]{Har77}) that for $n\le 8$, the divisor $-K_n$ is ample, hence
 $\Mor(X)\subseteq \Mor(X)^\prec$ and so
  $\Mor(X)=R_n$ by Mori's theorem. These are \emph{Del Pezzo} surfaces.
 As there are only finitely many $(-1)$-curves on $X$, the cone
 $\Mor(X)$ is polyhedral.
 If $\kappa_n=\co(3;1^n)$ is the \emph {anticanonical ray},
 then $\kappa_n$ is in the interior of the nonnegative cone $\cQ_n$.

\begin{exercise}
Show that, if $n \ge 9$, there are infinitely many $(-1)$-curves on
$X$ and infinitely many Cremona maps whose indeterminacy locus
is contained in $\{p_1,\dots,p_n\}$. Deduce that the locus of sets of
points in Cremona-general position is dense but not Zariski open
in $(\bbP^2)^n$.
\end{exercise}

When $n=9$, the anticanonical divisor $-K$ is an irreducible curve with self-intersection 0.
Hence $\kappa$ is nef, sits on  $\partial\cQ$,
and the tangent hyperplane to $\partial\cQ$ at $\kappa$
is the hyperplane $\kappa^\perp$  of classes $\xi$ such that
$\xi\cdot K=0$.
Then $\Mor(X)^\succcurlyeq=\kappa$ and
$\Mor(X)=\kappa+R\subseteq   \Mor(X)^\preccurlyeq$.
The infinitely many $(-1)$-curves on $X$
determine infinitely many $(-1)$-rays,  and $\kappa$ is the only limit ray
of the $(-1)$-rays.
The anticanonical ray $\kappa_9$ coincides with the Nagata ray $\nu_9$.

 Since the classes of all $(-1)$ curves must belong to every
 system of generators of $\Eff X$, it follows that $\Eff X$
 and the total coordinate ring $\mathcal{TC} (X)$ are not finitely
generated as soon as $n\geq 9$. In this case, $\bbC[\bfx,\bfy]^K$
is a counterexample to Hilbert's 14-th problem.
In fact, it is possible to exhibit explicit
configurations of points whose blowup contains
infinitely many $(-1)$-curves. The interested reader can
find details in Mukai \cite{Muk04} and Totaro \cite{Tot08}
(see also Exercise \ref{points-cubic} below). Mukai
shows, more generally, that when the points $p_1,\ldots,p_n$
are sufficiently general in $\bbP^{r-1}$, the inequality
$n\geq \frac{r^2}{r-2}$
is a sufficient condition for the existence of infinitely many negative
divisors, and hence for non finite generation of $\bbC[\bfx,\bfy]^K$.
Totaro gives specific sets of points that work
over every field (including finite fields).

For $n\ge10$, the shape of $\Mor(X)$ is not well known.
$-K_n$ is not effective, and has negative self-intersection $9-n$.
Hence $\kappa_n$ lies off the nonnegative cone $\cQ_n$, which in turn has non-empty
 intersection with both $\Mor(X)^\succ$ and $\Mor(X)^\prec$.
 The rays spanned by the infinitely many $(-1)$-curves on $X$
 lie in $\Mor(X)^\prec$ and their limit rays
lie at the intersection of $\partial\cQ$ with the hyperplane
 $\kappa^\perp$.
  The Nagata ray $\nu$ sits on $\partial\cQ^\succ$.
  The plane  joining the rays $\kappa$ and $\nu$ is the
   \emph{homogeneous slice}, formed by the classes of \emph{homogeneous
    linear systems} of the form $(d;m^n)$, with $d\ge 0$.

\subsection {Nagata-type statements for extremal rays}
\label{ssec:more}

Nagata's conjecture states a necessary condition for the linear
system $\cL=(d;m_1,\dots,m_n)$ to be nonempty. In fact there
is a stronger conjecture that posits what the
dimension of $\cL$ should be, and that has also been
open for decades, namely the Segre-Harbourne-Gimigliano-Hirschowitz
(or SHGH) conjecture
(see  \cite {Seg62}, \cite{Har86}, \cite{Gi87}, \cite{Hir89} \cite{CM01}, quoted in chronological order).

\begin{conjecture}[SHGH Conjecture]
Let $d\ge 0, m_i\ge 0$ be such that $(d;m_1,\dots,m_n)\cdot E\ge0$
for every $(-1)$-curve $E$. Then
\begin{equation}\label{nonspecial}
\dim \left| dL-\sum_{i=1}^n m_iE_i\right|=
\max \left\{-1,\frac{d(d+3)}2-\sum_{i=1}^n\frac{m_i(m_i+1)}2 \right\}.
\end{equation}
\end{conjecture}

This conjecture expresses the expectation that the conditions
imposed by the multiple points should be independent, except
when the system meets negatively a $(-1)$-curve, the only
known case in which the conditions become dependent.
One can compute which linear systems are expected to be
nonempty according to the SHGH conjecture, and
obtain the following conjecture, first proposed by
T.~de~Fernex in \cite {dF10}.

\begin {conjecture}[De Fernex conjecture]
\label{conj:df} If $n\ge 10$, then
\begin{equation}\label{eq:df}
\Mor(X)=\cQ^\succcurlyeq \;+\; R,
\end{equation}
where $\cQ$ is the nonnegative cone \eqref{nonegcone} and   $R$ is the negative part in the Mori decomposition of $\Mor(X)$ ( as in \eqref{Raggin}).
\end{conjecture}

Let \[
D_n=(\sqrt {n-1}, 1^n)\in N^1(X)
\]
 be the \emph{de Fernex} class
  and $\delta_n=\co(D_n)$ the corresponding ray. One has $D_n^2=-1$, and  $D_n\cdot K_n=n-3\sqrt{n-1}=\frac{n^2-9n+9}{n+3\sqrt{n-1}}>0$
for $n\geq 8$ and, if $n=10$, one has $D_n=-K_n$.
  Set
 \begin{equation}\label{Delta}
 \begin{split}
  &\Delta_n^\succcurlyeq = \left\{\xi\in N^1(X) \text{such that }  \xi\cdot D_n\ge 0\right\}
  \\
&\Delta_n^\preccurlyeq = \left\{\xi\in N^1(X) \text{such that }  \xi\cdot D_n\le 0\right\}.
 \end{split}
 \end{equation}

\begin{theorem}[de Fernex \cite{dF10}]\label{thm:df} If $n\ge 10$ then:
\begin{itemize}
\item [(i)] all $(-1)$-rays lie in the cone $\cD_n:=\cQ_n-\delta_n$;
\item [(ii)] if $n=10$, all $(-1)$-rays lie  on the boundary  of the cone $\cD_n$;
\item [(iii)] if $n>10$, all $(-1)$-rays lie  in the complement of the cone $\kappa_n:=\cQ_n-\kappa_n$;
\item [(iv)] $\Mor(X)\subseteq \overline{\kappa+R}$;
\item [(v)] if  Conjecture \ref {conj:df} holds, then
\begin{equation}\label{eqdf2}\Mor(X) \cap \Delta_n^\preccurlyeq=\cQ_n\cap \Delta_n^\preccurlyeq.\end{equation}
\end{itemize}
\end{theorem}

\begin{remark} As noted in \cite {dF10}, Conjecture \ref {conj:df} does not imply  that $\Mor(X)^\succcurlyeq=\cQ^\succcurlyeq$, unless $n=10$, in which case this is exactly what it says (see Theorem \ref {thm:df}(v)). Conjecture \ref {conj:df} does imply Nagata's conjecture, and
is consistent with what is known about the boundary of $\Mor(X)$,
like Theorem \ref{thm:chmr}.
\end{remark}

\begin{proof}
All statements except (iv) are computations that can be
done as exercises; \cite{dF10} contains all the details.
For (iv), de Fernex uses a specialization argument, showing
that the claim holdes when the points are very general on an
irreducible cubic curve; if there exists an effective
integral class outside $\kappa+R$ for very general
position of the points, then it also exists for points
on a cubic curve. Note that an irreducible cubic through
all points has class $(3;1^n)=-K_n$, that is fixed by
arithmetic Cremona transformations, and it follows from this
that very general choices of points on an irreducible cubic
are Cremona general (see Exercise \ref{points-cubic} below);
so $R_n$ stays the same before and after
specializing to the cubic.

For points on an irreducible cubic, the strict transform
of this cubic is the only irreducible curve
$C$ with $[C]\in \Mor(X)^\succ$. On the other hand, there
are no irreducible curves $D$ with $D\cdot K_n=D\cdot C=0$,
because such a curve would have class $(d;m_1, \dots,m_n)$
with $3d=\sum m_i$, so
\[\cO_X(D)|_C =\cO_C(dL|_C-m_1 p_1-\dots -m_n p_n)\]
would be an effective line bundle of degree 0, that by the choice
of the $n\ge 10$ points would be general in $\Pic C$, therefore
non-effective ($C$ is of genus 1), a contradiction.

Thus $\NE(X)\subseteq \kappa_n+R_n$ and the claim follows.
\end{proof}

\begin{exercise}\label{points-cubic}
Show that if $C$ is an irreducible cubic curve with multiplicity 1
at each of $n$ points $p_1, \dots, p_n$, then its transform
by a standard Cremona map based at any three of the $n$ points
is again an irreducible cubic with multiplicity 1
at each of the resulting $n$ points. Deduce that for every $n$ there
are subsets of $n$ points in $C$ that are Cremona general.
\\
\emph{Hint}: The Cremona transform of a line through 3 points
cuts on $C$ an effective divisor of degree 0.
\end{exercise}

\begin{conjecture}[$\Delta$-conjecture, \cite{CHMR13}] \label{conj:CHMR}
Let $\cQ_n$ be the nonnegative cone \eqref{nonegcone} and $\Delta_n^\preccurlyeq$ as in \eqref{Delta}. If $n\ge 10$ then
\begin{equation}\label{eq:dCHMR}
\partial\cQ_n\cap \Delta_n^\preccurlyeq\subset \Nef(X).\end{equation}
\end{conjecture}

\begin{proposition}\label{prop:iml}  If the $\Delta$-conjecture holds, then
\begin{equation}\label{eq:dCHMR2}
\Mor(X) \cap \Delta_n^\preccurlyeq=\Nef(X) \cap \Delta_n^\preccurlyeq=\cQ_n\cap \Delta_n^\preccurlyeq.
\end{equation}
\end{proposition}

\begin{proof}   By \eqref {eq:dCHMR} and  by
convexity of $\Nef(X)$ one has
$$\cQ_n\cap \Delta_n^\preccurlyeq\subseteq \Nef(X) \cap \Delta_n^\preccurlyeq.$$
Moreover  $\Nef(X) \cap \Delta_n^\preccurlyeq\subseteq \Mor(X) \cap \Delta_n^\preccurlyeq$. Finally \eqref {eq:dCHMR} implies  \eqref {eqdf2}   because $\Mor(X)$ is dual to $\Nef(X)$.
\end{proof}

The following proposition shows that Nagata-type conjectures we are discussing here can be interpreted as asymptotic forms of the SHGH conjecture.

\begin{proposition}\label{prop:ANS} Let $n\ge 10$.
If the $\Delta$-conjecture holds, then all  classes  in $\cQ_n\cap \Delta_n^\preccurlyeq- \partial\cQ_n\cap \Delta_n^\preccurlyeq$ are ample and therefore, if integral, there is an integer $y$ such that for all nonnegative integers $x\ge y$ the dimension of $(xd;xm_{1},\ldots, xm_{n})$
is given by \eqref{nonspecial}.
\end{proposition}

\begin{proof} It follows from Proposition \ref {prop:iml} and  the fact that the ample cone is the interior of the nef cone (by Kleiman's theorem, see \cite {Kle66}).
\end{proof}

One can give a stronger form of the $\Delta$-conjecture.

\begin{lemma}\label{lem:nef} Any rational, non-effective ray in $\partial\cQ_n$ is nef and it is extremal for both $\Mor(X)$ and $\Nef(X)$. Moreover it lies in $\partial\cQ_n^\succcurlyeq$.
\end{lemma}

\begin{proof}
That such a ray is nef and extremal for $\Mor(X)$ was proved
in Exercise \ref{exercise-extremal}. The duality between $\Mor(X)$ and $\Nef(X)$
shows that it is also extremal in $\Nef(X)$.
The final assertion  follows by Mori's cone theorem.
\end{proof}

A rational, non-effective ray in $\partial\cQ_n$  will be called a \emph{good ray}. An  irrational, nef ray in $\partial\cQ_n$  will be called a \emph{wonderful ray}. No wonderful ray has been detected so far.

The following conjecture  implies the $\Delta$-conjecture.

\begin{conjecture}[Strong $\Delta$-conjecture] \label{conj:SCHMR} If $n>10$, all rational rays in $\partial\cQ_n\cap \Delta_n^\preccurlyeq$  are non-effective.  If $n=10$, a rational ray in  $\cQ_{10}\cap \Delta_{10}^\preccurlyeq=\cQ_{10}^\succcurlyeq$ is non--effective, unless it is generated by a Cremona
transform of the curve with class  $(3;1^9, 0)$.
\end{conjecture}

\begin{proposition} \label{prop:SN10} For $n=10$,
the strong $\Delta$-conjecture
is equivalent to the following statement (``Strong Nagata conjecture''):
If $C$ is an irreducible curve of genus $g>0$ on $X$, then $C^2>0$ unless
$n\ge 9$, $g=1$ and $C$ is a Cremona transform of the curve with class
$(3;1^9, 0^{n-9})$, in which case $C^2=0$.
\end{proposition}

\begin{proof}
If  the strong $\Delta$-conjecture holds, then clearly the Strong Nagata conjecture
holds. Conversely, consider a
rational effective ray in   $\partial\cQ_{10}^\succcurlyeq$ and let $C$ be an effective divisor
in the ray. Then $C=n_1C_1+\cdots +n_hC_h$, with $C_1,\ldots, C_h$ distinct irreducible curves
and $n_1,\ldots, n_h$ positive integers. One has  $C_i\cdot C_j\ge 0$, hence
$C_i\cdot C_j=0$ for all $1\le i\le j\le h$. This  clearly implies $h=1$, hence the assertion.  \end{proof}

By the proof of Proposition \ref{prop:iml}, any good ray  gives
a constraint on $\Mor(X)$, so it is useful to find  good rays.  Even better would be to find
wonderful rays.

\begin{example}\label{ex:seq}  Consider the family of linear systems
\[ \cB=\{ B_{q,p}:= (9q^2+p^2; 9q^2-p^2, (2qp)^9):   (q,p)\in \bbN^2, q\le p \}\]
generating rays in $\partial\cQ_{10}^\succcurlyeq$. Take a sequence $\{(q_n,p_n)\}_{n\in \bbN}$ such that $\lim_n \frac {p_n+q_n}{p_n}=\sqrt {10}$. For instance take $\frac {p_n+q_n}{p_n}$ to be the convergents of the periodic continued fraction expansion of $\sqrt {10}=[3; \overline{6}]$, so that 
\[ p_1=2, \; p_2=13, \; p_3=80,\ldots \;\; q_1=1, \; q_2=6, \; q_3=37, \ldots. \]
The sequence of rays $\{[B_{q_n,p_n}]\}_{n\in \bbN}$ converges to the Nagata ray $\nu_{10}$. If we knew that the rays of this sequence are good, this would imply
Nagata's conjecture for $n=10$.
\end{example}

\subsection{When does finite generation hold?}\label{polyhedral}

We have seen that the blow-up
$f: X=X_n\to \bbP^2$ of the plane at very general points $p_1,\ldots, p_n$
has finitely generated (i.e., polyhedral) Mori cone $\Mor(X)$ if and only
if $n\le 8$. Although the main focus of these notes is on Nagata type rays,
i.e., on non finitely generated cases, it should be mentioned that
characterizing the sets of points $p_1,\dots,p_n\in\bbP^2$ such that the Mori cone
(respectively, the effective semigroup $\Eff X$, the Cox ring $\mathcal{TC}(X)$)
of the blow-up is finitely generated, and studying these particular surfaces,
is an important and active area of research.
With no attempt at being comprehensive, we now review
a few results in this area.
In this section we drop the assumption that the points $p_i$ are general.

On any blowup of $n\le 9$ points, the anticanonical divisor $-K=(3;1^n)$
is effective. This puts great restrictions on curves $C$ with negative
selfintersection; namely, by adjunction we have that the genus $g$
of such a curve satisfies
\[
C^2+KC =2g-2 \ge -2,
\]
so the inequality $(-K)\cdot C\ge 0$ (which holds unless $C$ is a
fixed component of $|-K|$) implies that $C^2\ge -2$ and every
curve with negative selfintersection is rational. Observe that this
will continue to hold for $n\ge 10$ points, as long as the anticanonical
divisor is effective, i.e., the points lie on a (possibly reducible)
cubic curve.
Note that this is essentially the same idea used in the proof
of Theorem \ref{thm:df}, and it will be thoroughly exploited
in the third section.

If $n\le 8$, or more generally, if $\dim |-K|>0$,
every curve $C$ that is not a fixed component of $|-K|$ must
have $(-K)\cdot C>0$. In this case the only curves with negative
selfintersection are the fixed components of $|-K|$ and the
$(-1)$-curves.

Using these facts,
it is not hard to prove the following:

\begin{proposition}
The blow-up $X$ of $\bbP^2$ at an arbitrary set of $n\le 8$ points
or at a set of $n\ge 9$ points lying on a conic has
finitely generated $\Eff X$ and $\Mor(X)$.
\end{proposition}

In fact under the conditions of the proposition more can be said:
B.~Harbourne computed the dimension of all linear systems
$(d;m_1,\dots,m_n)$, i.e., the Hilbert functions of all
$I(Z_\bfm)$, and even their graded free resolutions, in \cite{Har98}.

In the cases when $-K$ is effective
(or some multiple $-mK$ is effective, i.e., on a Coble surface)
but fixed, there is to the best of our knowledge no
complete characterization of the sets of points that give
finitely generated Mori cones; see however \cite{AL11}, \cite{CD12},
\cite{CT15}, \cite{GM16} and references therein.
A few of these works care also about the finite generation of
the total coordinate ring $\mathcal{TC}(X)$; this is in itself
an  interesting problem,
and it turns out that there are special blow-ups of
$\bbP^2$ where no multiple of $-K$ is effective and yet the
total coordinate ring is finitely generated \cite{GM16}.

Finally, let us also mention that by a result of Nikulin \cite{Nik00},
the surfaces with polyhedral Mori cone whose generating
curves have bounded degree and genus can be classified.

\section{Conjectures on valuations}

\subsection{Valuations and good rays}
We now move to a slightly different setting, namely blowups of $\bbP^2$
determined by some particular valuations, and finite generation
questions on them.
We refer to the references
O.~Zariski--P.~Samuel \cite[Chapter VI. and Appendix 5.]{ZS75II} and
E.~Casas--Alvero \cite[Chapter 8]{Cas00}
for the general theory of valuations and complete ideals on surfaces.

A rank 1 valuation
on a domain $R$ is a map
\[v:R\rightarrow \bbR\cup \{\infty\}\]
satisfying
\begin{equation}\label{def:valuation}
 v(fg)=v(f)+v(g), \qquad v(f+g)\ge\min(v(f),v(g)), \qquad
v(f)=\infty \Leftrightarrow f=0,
\end{equation}
for all $f, g\in R$.
Note that a valuation on a domain $R$ determines a unique valuation
on its quotient field $K$ by setting $v(f/g)=v(f)-v(g)$, and
conversely a valuation on a field $K$ restricts to a valuation
on any subring $R\subset K$. The \emph{value group}
of $v$ is $v(K^*)\subset \bbR$, a subgroup of the additive group of $\bbR$.
We will be mostly interested in the case $K=\bbC(x,y)$, and we
only consider valuations with trivial restriction to $\bbC$, that is
$v(w)=0 \ \forall w\in \bbC$.

Given a valuation $v:K\rightarrow \bbR\cup \{\infty\}$, the set of
elements $f\in K$ with $v(f)\ge0$ is a subring
$R_v\subset K$ called the \emph{valuation ring} of $v$. Valuation rings
are characterized as those subrings $S\subset K$ such that, for
every $f\in K$, either $f\in S$ or $f^{-1}\in S$.
Every valuation ring $R_v$ is a local ring, with maximal ideal $\fm_v$
consisting of those elements with positive value. Except
when the value group is discrete (i.e., there is $a\in \bbR$ such that
$v(K^*)=\bbZ a$), valuation rings are not noetherian.

\subsubsection*{Valuation ideals and volume}
We are interested in valuations on the field $\bbC(x,y)$ of
rational functions on $\bbP^2$.
Choose homogeneous coordinates
$w_1,w_2,w_3$ on $\bbP^2$, in such a way that $x=w_2/w_1, y=w_3/w_1$.
Given a valuation $v$ on  $\bbC(x,y)$ it is possible to extend it to a nonnegative valuation $v$ on the ring  $R=\bbC[w_1,w_2,w_3]$ as follows. If $v(x)\ge0$, $v(y)\ge0$,
then one simply sets $v(F_d(w_1,w_2,w_3))=v(F_d(1,x,y))$.
Otherwise let $v_{\min}=\min(v(x),v(y))<0$, and 
set 
\[v(F_d(w_1,w_2,w_3))=v(F_d(1,x,y))-d v_{\min}.\]
In particular for instance $v(w_1)=-v_{\min}$ and 
$\min\{v(w_1),v(w_2),v(w_3)\}=0$. Then the definition
extends to nonhomogeneous polynomials as
\(v(F)=\min\big\{v(F_d)\big\}\)
for any $F=\sum F_d$ in $R$, where $F_d$ is the
 homogeneous degree $d$
part of $F$.
For every non-negative $m\in \bbR$, the homogeneous ideals
\[
I_m=\{F\in R\,|\, v(F)\ge m \}, \quad \text{and} \quad
I_m^+=\{F\in R\,|\, v(F)> m \}
\]
are called \emph{valuation ideals}. They form multiplicative filtrations, that is
$I_m^+\subset I_m \subset I_{m'}^+\subset I_{m'}$ whenever $m'>m$,
moreover
$ I_mI_{m'}\subset I_{m+m'}$, and $ I_mI_{m'}^+\subset I_{m+m'}^+$.
Recall from previous sections the notation
$\alpha(I)=\min\{d|I_d\ne 0\}$ whenever $I$ is a graded ideal,
and consider the number
\[ \mu_d(v)=\max\{m \in \bbZ\,|\, [I_m]_d\ne 0\}
=\max\{m \in \bbZ\,|\,\alpha(I_m)\le d\}.\]
\begin{exercise}\label{limits}
 The limits
\[ \widehat \alpha(v)=\lim_{m \to \infty}\frac{\alpha(I_m)}{m}, \qquad
\widehat \alpha^+(v)=\lim_{m \to \infty}\frac{\alpha(I_m^+)}{m},\qquad
\widehat\mu(v)=\lim_{d\to\infty}\frac{\mu_d(v)}{d} \]
 exist and $\widehat \alpha(v)=\widehat \alpha^+(v)=\widehat \mu(v)^{-1}$. The number
 $\widehat{\alpha}(v)$ is called the \emph{Waldschmidt constant} of $v$
\\
 Hint: Look at Exercise 1.3.3(d) in  Harbourne's notes \cite{Har16}.
\end{exercise}

The description above of valuation ideals on $\bbP^2$ is a particular instance
of a more general construction. Let $X$ be a projective algebraic variety and
$v:K(X)\rightarrow \bbR \cup \{\infty\}$ a rank 1 valuation on the field
of rational functions of $X$.
Then, for every nonnegative $m\in \bbR$ one has valuation ideal sheaves
\[
\cI_m=\left(f\in \cO_X\,|\, v(f)\ge m \right), \quad \text{and} \quad
\cI_m^+=\left(f\in \cO_X\,|\, v(f)> m \right),
\]
and for every divisor class $D$, graded ideals $I_m, I_m^+$
in the graded ring $\oplus_{k\ge 0} H^0(X,kD)$.
The definitions of $\widehat\alpha$, $\widehat\alpha^+$ and
$\widehat\mu$ also carry over to this setting.

\begin{exercise}\label{sheaf}
Work out the details of the previous sheaf-theoretic definitions.
More precisely:
\begin{enumerate}
 \item For every affine open set $U\subset X$, let
 $R_U=\Gamma(\cO_X,U)$. Check that the valuation $v$
 restricts to a valuation of $R_U$.
 \item If there is $f\in R_U$ with $v(f)<0$ then set $I_{U,m}=I_{U,m}^+=R_U$
 for every $m\ge 0$;
 otherwise the set of elements in $R_U$ with value greater than or equal to
 (respectively, greater than) $m$ is an ideal $I_{U,m}$
 (respectively, $I_{U,m}^+$) of $R_U$.
 \item Gluing: the data $U \mapsto I_{U,m}$ (respectively  $U \mapsto I_{U,m}^+$)
 define a subsheaf $\cI_m$ (respectively $\cI_m^+$) of the structure sheaf $\cO_X$.
 \item \label{prime} If there is no $f\in R_U$ with $v(f)<0$ then
 $I_{U,0}^+$ is a proper prime ideal of $R_U$.
 \end{enumerate}
\end{exercise}

By part \ref{prime} of Exercise \ref{sheaf}, the sheaf $\cI_0^+$ determines
an irreducible proper subvariety of $X$, called \emph{the center} of the
valuation $v$ on $X$, and denoted by $\cent_X(v)=\cent(v)$. Set
$R_v=\{f\in K(X)\,|\,v(f)\ge 0\}$   the valuation ring of
$v$, the generic point $\eta=\eta_{\cent_X(v)}$ of the center
is the image of the closed point of $R_v$ under the unique
map $\Spec R_v \rightarrow X$ that exists by the valuative
criterion of properness \cite[II.4.7]{Har77}.
Therefore $\cent(X)$ is nonempty, and $v$ is nonnegative
on the local ring $\cO_{X,\eta}$.

\begin{exercise}
Work out the details of the graded valuation ideals, and check that for
the plane $\bbP^2$ they agree with the former definitions. More precisely:
\begin{enumerate}
\item
 Every trivializing open subset $U\subset X$ for $\cO_X(D)$
is also trivializing for $\cO_X(kD)$.
\item
 Via the induced maps
$\Gamma(\cO_X(kD),U)\overset{\sim}{\rightarrow}\Gamma(\cO_X,U)$, the valuation
$v$ determines a valuation $v_{D,U}$ on $R(D)=\oplus_{k\ge 0} H^0(X,kD)$.
\item
 If $U$ is a neighborhood of $\cent(v)$, then the valuation
$v_{D,U}$ is nonnegative on $R(D)$, independent on the choice of $U$.
Denote it by $v_D$.
\item For every $m$,  the spaces $I_{k,m}=H^0(X,\cI_m\otimes \cO_X(kD))$
are the graded pieces of an ideal $I_m=\oplus_{k\ge0}I_{k,m} \subset R(D)$,
and
\[ I_m=\left\{s\in \underset{k\ge 0}{\mathsmaller{\bigoplus}} H^0(X,kD)\,|\, v_D(s)\ge m\right\}.\]
\end{enumerate}
\end{exercise}

In particular, by the preceding exercise, given any Cartier divisor
$D$ on $X$ its valuation $v(D)$ is well defined: it equals the
valuation of any local equation of $D$ on a neighborhood of $\cent_X(v)$.
We use this fact without further mention in the sequel.

For a valuation $v$ with zero-dimensional center
on an $n$-dimensional variety $X$, the volume
was defined in \cite{ELS03} as
$$\vol(v) \,:=\,\lim_{m\to\infty}
\frac{\dim_{\bbC} (\cO_X/\mathcal{I}_m)}{m^n/n!}$$
(note that $\cO_X/\mathcal{I}_m$ is an artinian
$\bbC$-algebra supported
at the center of the valuation).
On the other hand, the volume of a divisor class $D$ on $X$
is defined as
\[
\vol (D) := \limsup_{k\to\infty} \frac{h^0(S, kD)}{k^n/n!} .
\]
Boucksom-K\"uronya-MacLean-Szemberg \cite{BKMS} show that the limit
\[\widehat\alpha_D(v)=\lim_{m\to\infty}\frac{\min\{k\in\bbZ\,|\,I_{k,m}\ne 0\}}{m} \]
exists (generalizing Exercise \ref{limits}) and
can be bounded in terms of volumes:
\begin{proposition}[{\cite[Proposition 2.9]{BKMS}}]\label{volBKMS}
Let $D$ be a big divisor and $v$ a real valuation centered at a point
$p\in X$. Then
\begin{equation*}
 \widehat \alpha_D (v)\le \sqrt[n]{\vol(v)/\vol(D)}\ .
\end{equation*}
\end{proposition}

When $D$ is ample this bound is equivalent to
\(
 \widehat\alpha_D(v)\le\sqrt[n]{\vol(v)/D^n} \ .
\)
The interested reader will find in section \ref{sec:cluster} below
a hint (Exercise \ref{ex:muvol}) for the proof of this result in the particular cases
of interest to us.
Valuations satisfying the equality in Proposition \ref{volBKMS}
will be called \emph{maximal}.

We are especially interested in finding
maximal valuations with respect to a line $D=L\subset X=\bbP^2$.
Analogously to the previous sections, we may consider the support
semigroup
\[Supp_v(\mathsmaller{\bigoplus} H^0(X,kD))=\{(k,m)\in \bbZ^2 \,|\, I_{k,m}\ne0\}
\subset \bbZ^2 \subset \bbR^2\]
and the cone spanned by it:
\[
\co(v(D))=\co\left(Supp_v
\left(\mathsmaller{\bigoplus} H^0(X,kD)\right)\right)\subset \bbR^2 \ .
\]
As a planar cone, $\co(v(D))$ has two boundary rays: $\co(1,0)$ and
$\co(\widehat{\alpha}_D(v),1)$. If the valuation $v$ is maximal, the latter
may be a \emph{good ray},
that is, it may happen that
\[ \co(\widehat{\alpha}_D(v),1) \subset
\overline{\co(v(D))}\setminus \co(v(D)),\]
and in that case $v(s)<k/\widehat{\alpha}_D(v)$ for all
$s\in H^0(X,kD)$, i.e., a Nagata-type statement holds. Hence
our interest in maximal valuations on the projective plane.
\subsection{The space of valuations with given center}
\label{space_of_valuations}

If $X$ is a surface and $v$ is a valuation on $K(X)$,
whose center is not a closed point, then either $\cent(v)=X$, in which case
$v$ is the trivial valuation ($v(f)=1 \ \forall f \ne 0$) or
$\cent(v)=C$ is a curve. In the latter case, let
$p\in C\subset X$ be any point on $C$
and assume $f\in \cO_{X,p}$ is a germ of equation for $C$. Then
$v$ is non-negative on $\cO_{X,\eta}$, where
$\eta$ is the generic point of $C$, and hence on
$\cO_{X,p}\subset\cO_{X,\eta}$,
and therefore $v(u)=0$ for every invertible element $u$ of
$\mathcal{O}_{X,p}$. For any $g\in \mathcal{O}_{X,p}$
one can write $g=g'f^s$ for some $g'$ invertible in
$\mathcal{O}_{X,p}$ and some non-negative integer $s$, and therefore
$v(g)=sv(f)$. Thus, whenever $\cent(v)$ is a curve $C$, the valuation $v$ is
(up to a constant $c=v(f)\in \bbR$)  the order of vanishing
along $C$; i.e., for every divisor $D$, one has
$v(D)=c \cdot \ord_C D=c \cdot \max\{k\,|\,D-kC\ge 0\}$.
These are called \emph{divisorial valuations}.

Henceforth we focus in the case that the center of $v$ is a closed point $p\in X$.
Such valuations are non-negative on the local ring
$\cO_{X,p}$, i.e., they restrict to
maps $v:\cO_{X,p}\rightarrow \bbR_{\ge 0} \cup\{\infty\}$
satisfying \eqref{def:valuation}.
The minimal strictly positive value of $v$ on $\cO_{X,p}$ is called the
\emph{value}  of $v$ at $p$, $v(p)$; it is the common value of
\emph{general} elements in the maximal ideal $\mathfrak{m}_{X,p}\subset \cO_{X,p}$
\cite[8.1]{Cas00}. An example of a valuation with zero-dimensional
center is the order
of vanishing at $p$, that can be also obtained blowing up $X$ at $p$,
and  considering  the divisorial valuation  centered on the exceptional
divisor.

\begin{example}[Monomial valuations]
Fix affine coordinates $(x,y)$ near
$\cent(v)=p=(0,0)\in\bbA^2=\Spec \bbC[x,y] \subset \bbP^2=\Proj \bbC[w_1,w_2,w_3]$,
with $x=w_2/w_1, y=w_3/w_1$.
Given two nonnegative real numbers $s,t$, we can define a valuation on
$\bbC[x,y]$ by setting
\[
v_{s,t}\left(\sum_{\substack{i,j\ge0\\i+j\le d} }
a_{ij}x^i y^j\right)=\min\{si+tj | a_{ij}\ne0\}.
\]
As particular cases we obtain that $v_{0,0}$ is the trivial valuation;
$v_{s,0}$ is $s$ times the divisorial valuation centered on the line $x=0$; while
$v_{0,t}$ is $t$ times the divisorial valuation centered on the line $y=0$; and
$v_{1,1}$ is the order of vanishing at $p$.
Whenever $s\cdot t>0$, the center of $v_{s,t}$ is $p=(0,0)$.
\end{example}

Remark that for every $\lambda>0$, one has
$v_{\lambda s,\lambda t}=\lambda v_{s,t}$, hence
there is an equality of valuation rings
$R_{v_{\lambda s,\lambda t}}=R_{v_{s,t}}$.
Two valuations $v$, $v'$ with the same valuation ring
are called \emph{equivalent}.

\begin{example}[Quasimonomial valuations]
Let $u,w\in \bbC[x,y]$ be a system of parameters for $p$, i.e.,
\[
(u,w)\cO_{X,p}=\fm_{X,p},
\]
or in other words, the curves $\{u=0\}$ and  $\{w=0\}$ meet transversely
at $p=(0,0)$. Then, every element $f$ in $\cO_{X,p}$ (or in its completion
$\widehat{\cO_{X,p}}$, or in the polynomial ring $\bbC[x,y]$)
has a Taylor expansion
\[
f=\sum_{i,j\ge 0} a_{ij}u^iw^j,
\]
and we can define
\(
v_{s,t}^{u,w}(f)=\min\{si+tj | a_{ij}\ne0\}.
\)
Again one obtains as extreme cases the divisorial valuations
associated to the curves $\{u=0\}$ and  $\{w=0\}$, and for positive parameters
the valuations obtained have center at $p$. Note that this construction
is possible on every smooth point of a surface $X$.
\end{example}

\begin{exercise}
 Assume $s,t>0$ and let $\cI_{m}$ be the valuation
 ideal  with respect to the valuation $v_{s,t}^{u,w}$. Show that
 $\cI_m$ has cosupport at the point $p$.
Let $I_{m,p}$ be the stalk at $p$ of $\cI_{m}$.
 Show that the set of classes $\{[u^i w^j]\}_{si+tj<m}$
 form a basis of
 $\cO_{X,p}/I_{m,p}$ as a $\bbC$-vector space.
 Deduce that $\vol v_{s,t}^{u,w}=1/st$.
\end{exercise}

\begin{proposition}\label{sufficiency}
Let $u_1, u_2, w_1,w_2\in \cO_{X,p}$ and $t>s>0$. Assume that
\begin{enumerate}
 \item $(u_1,w_1)=(u_1,w_2)=(u_2,w_1)=(u_2,w_2)=\fm_{X,p}$, i.e., each pair
$(u_i,w_j)$ is a system of parameters;
 \item $\dim_\bbC \cO_{X,p}/(w_1,w_2)\ge t/s$.
\end{enumerate}
Then $v_{s,t}^{u_1,w_1}=v_{s,t}^{u_1,w_2}=v_{s,t}^{u_2,w_1}=v_{s,t}^{u_2,w_2}$.
\end{proposition}
 Note that the second hypothesis means that the intersection multiplicity
 of $\{w_1=0\}$ and $\{w_2=0\}$ at $p$ is at least $t/s$; given that both $\{w_1=0\}$ and $\{w_2=0\}$ are smooth germs of curves at $p$ by the first hypothesis,
 this is equivalent to
 saying that the $\lceil t/s\rceil$-jets of $w_1$ and $w_2$ coincide, i.e.,
 $w_1-w_2\in \fm_{X,p}^{\lceil t/s\rceil}$. Thus Proposition \ref{sufficiency}
says that whenever $t>s$, the valuation   $v_{s,t}^{u,w}$ does not
depend on the choice of $u$, and it only depends on the
 $\lceil t/s\rceil$-jet of $w$.
\begin{proof} [Proof of Proposition \ref{sufficiency}]
  By the first hypothesis, there is a series $h(w_1)=\sum_{i\ge1}a_i w_1^i$
  with $u_2=u_1+h(w_1)$. Since $t>s$, then
  $v_{s,t}^{u_1,w_1}(h(w_1))\ge v_{s,t}^{u_1,w_1}(w_1)>v_{s,t}^{u_1,w_1}(u_1)$, and
  \[ v_{s,t}^{u_1,w_1}(u_1)=
 v_{s,t}^{u_1,w_1}(u_1+h(w_1))=v_{s,t}^{u_1,w_1}(u_2).\]
 Therefore, for every $f\in \cO_{X,p}$, the Taylor expansions of $f$
 with respect to $(u_1,w_1)$ and $(u_2,w_1)$ are related by
 \[
\begin{split}
 f=\sum_{i,j\ge 0} a_{ij}u_2^i w_1^j=
 \sum_{i,j\ge 0} a_{ij}(u_1+h(w_1))^i w_1^j=\\
 \sum_{i,j\ge 0} a_{ij}u_1^i w_1^j + \text{ terms with higher }v_{s,t}^{u_1,w_1}.
\end{split}
 \]
 By the definition
 of quasimonomial valuations, it follows that $v_{s,t}^{u_1,w_1}=v_{s,t}^{u_2,w_1}$
 and so also $v_{s,t}^{u_1,w_2}=v_{s,t}^{u_2,w_2}$.

 On the other hand, the second hypothesis implies that there is some
 series $h(u_1)=\sum_{i\ge \lceil t/s\rceil} a_i u_1^i$ with
 $w_2=w_1+h(u_1)$. As before, this implies that
 $v_{s,t}^{u_1,v_1}(w_1)=v_{s,t}^{u_1,v_1}(w_2)$, and plugging
 $w_2=w_1+h(u_1)$ into the Taylor series of any $f$, the equality
 $v_{s,t}^{u_1,w_1}=v_{s,t}^{u_1,w_2}$. We leave the details to the reader.
\end{proof}

\subsubsection*{Valuative trees}

Our next goal is to describe the space of all equivalence classes of
quasimonomial valuations of $\cO_{X,p}$, in the spirit of \cite{FJ04}.
In order to avoid dealing
with equivalent valuations, we
\emph{normalize} them in such a way that the minimum
strictly positive value of $f\in \cO_{X,p}$ is 1
(i.e., the value of $v$ at $p$, $v(p)=1$). For $v_{s,t}^{u,w}$,
this minimal value is $\min\{s,t\}$. Fix a system of parameters
$(x,y)\in \cO_{X,p}$ (for $X=\bbP^2$,
we set $x, y$ to be local affine coordinates).
Set
\begin{equation}\label{quasimonval}
\begin{split}
&\cQ =\left\{\text{ quasimonomial valuations
centered at } p\right\}/\text{equiv}\\
&\cQ_x =\left\{v\in \cQ \text{ such that } v(x)=v(p)\right\}\\
&\cQ_y =\left\{v\in \cQ \text{ such that } v(y)=v(p)\right\}.
\end{split}
\end{equation}
Note that $v(p)=\min\{v(x),v(y)\}$, so $\cQ=\cQ_x \cup \cQ_y$.

\begin{exercise}\label{ex:quasimonomial}
 Let $\xi=\xi(x)$ be a formal power series in $x$ and define,
for every $f\in\cO_{X,p}$,
\begin{equation}\label{series_valuation}
\vv[f]{\xi}{t} := \ord_x(f(x,\xi(x)+\theta x^t))\ ,
\end{equation}
where the symbol $\theta$ is transcendental over $\bbC$.
Show that $\vv{\xi}{t}$ is a valuation of $\cO_{X,p}$.

Let $w \in \cO_{X,p}$ be such
 that $w=0$ is not tangent to $x=0$ at the point $p=(0,0)$.
 Expand $w$ as a Taylor series or polynomial, $w=w(x,y)$. By the implicit
 function theorem, there is a convergent power series $\xi(x)$ such that
 $w(x,\xi(x))=0$.
Show that for this $\xi$, one has $\vv{\xi}{t}=v_{1,t}^{x,w}$.
\end{exercise}

\begin{theorem}\label{parametrized-tree}
Fix a system of parameters $(x,y)\in \cO_{X,p}$.
\begin{enumerate}
 \item For every $\xi=\xi(x)$ a formal power series in $x$, the map
\begin{align*}
 \bbR_{\ge 1}&\overset{v_\xi}\longrightarrow \cQ_x\\
 t & \longmapsto \vv{\xi}{t}
\end{align*}
is injective, and for every $f\in \cO_{X,p}$ the map
$t  \mapsto \vv[f]{\xi}{t}$ is continuous.
\item $\vv{\xi_1}{t_1}=\vv{\xi_2}{t_2}$ if and only if $t_1=t_2$
and $\ord_x(\xi_1-\xi_2)\ge t_1$.
\item For every $v\in \cQ_x$, there exist $\xi$ and $t$ such that
$v=\vv{\xi}{t}$.
\end{enumerate}
\end{theorem}

\begin{proof}
 We will show that $v_\xi$ is injective in the interval $[1,n]$ for every
 positive integer $n$. It follows that it is injective in the whole half line.
 Let $\xi(x)=\sum_{i=1}^\infty a_i x^i$ and consider
 \[ \omega_n= y-\sum_{i=1}^n a_i x^i\in\cO_{X,p}.\]
 An elementary computation shows that $v_{\xi,t}(\omega_n)=t$
 for $t\in [1,n]$, so the
 claimed injectivity follows.
 The rest of the claims are immediate consequences of previous results.
\end{proof}

\begin{remark}\label{rem:Newton}
For a fixed $f\in \cO_{X,p}$, the map $t  \mapsto \vv[f]{\xi}{t}$
is continuous, concave, piecewise linear
with integer coefficients (i.e., a tropical polynomial function).
To see this, let
\[\omega= y-\sum_{i=1}^\infty a_i x^i\in\widehat{\cO_{X,p}},\]
and expand $f$ as a power series
\[ f=\sum_{i,j\ge 0}a_{ij}x^iw^j\in\widehat{\cO_{X,p}}\, .\]
Let
\[S(f)=\operatorname{conv}\left(\left\{(i,j)\in \bbN^2\,|\, a_{ij}\ne 0\right\}\right)\]
be the convex hull of the support of $f$. Its lower left boundary is
called the \emph{Newton polygon} of $f$ (in the formal coordinates
$(x,w)$), and denoted by
\[N(f)=\partial\left(S(f)+(\bbR_{\ge 0})^2\right).\]
The Newton polygon $N(f)$ consists of a vertical half
line followed by a finite sequence of segments with
increasing negative (rational) slopes
and a horizontal half line. Let $\Gamma_1, \dots, \Gamma_k$
be the segments with slopes $\ge -1$, and call these slopes
 $-1\le \gamma_1 \le \dots\le\gamma_k$.
Let also $V_1, \dots, V_{k+1}$
be the vertices, so that $\Gamma_{\ell-1}\cap \Gamma_{\ell}=V_\ell$.

By Exercise \ref{ex:quasimonomial} we know that
$\vv{\xi}{t}(f)=\min\{i+tj\,|\,a_{ij}\ne0\}$, and clearly this
minimum is attained at at a monomial $a_{ij}x^iw^j$ with
$(i,j)\in N(f)$. Moreover the monomial is unique,
with $(i,j)$ one of the vertices $V_\ell$, unless $-t^{-1}$ is
the slope of one of the segments $\Gamma_{\ell}$. More precisely,
for all $t\in [-\gamma_{\ell-1}^{-1},-\gamma_\ell^{-1}]$, the minimum
is attained at $(i,j)=V_\ell$ (and $\vv{\xi}{t}(f)=i+tj$ in this interval, which is linear with integer slope).

In convex geometry the function $t\mapsto \vv{\xi}{t}(t)$ obtained
in this way is usually called the \emph{Legendre transform}
of the Newton polygon.
\end{remark}

We endow $\cQ_x$ with the final topology with respect to all maps
$v_\xi$. Because of the second statement in
Theorem \ref{parametrized-tree}, each of these maps becomes an homeomorphism
of the half-line $\bbR_{\ge 1}$ with its image, and the intersection
 \[v_{\xi_1}(\bbR_{\ge 1}) \cap v_{\xi_2}(\bbR_{\ge 1}) \]
 is homeomorphic to the segment $[1,\ord_x(\xi_1-\xi_2)]$.
It is easy to see that the topologies induced by $\cQ_x$
and $\cQ_y$in $\cQ_x\cap\cQ_y$ agree, endowing the whole
set $\cQ$ with a topology that makes it into a \emph{profinite $\bbR$-tree},
rooted at the valuation $v_{\xi,1}$ (which is the `order at $p$' valuation)
with maximal branches of the tree corresponding to the series $\xi$,
two branches separating at the points (of integer parameter $t$)
corresponding to $\ord_x(\xi_1-\xi_2)$.

\begin{remark}\label{tree}
 The tree of quasimonomial valuations just constructed is a subset of the
 \emph{valuative tree} $\cT$ of all classes of
 rank 1 valuations centered at $p$ introduced
 by Favre and Jonsson. To build the whole $\cT$ one proceeds in essentially
 the same way, observing that in \eqref{series_valuation}
  one may allow formal series $\xi(x)=\sum_{j\ge 1} a_j x^{\beta_j}$
whose exponents $\beta_j$ form
an arbitrary increasing sequence of rational numbers,
and one still obtains valuations $\vv{\xi}{t}$
(no longer quasimonomial). Unless the series defines an \emph{algebraic}
function, i.e., unless
it vanishes identically on some curve $C\subset X$, it is
also possible to allow $t=\infty$. The precise statement and proof of
Theorem \ref{parametrized-tree} then becomes technically more involved, see
\cite[Chapter 4]{FJ04} and \cite[8.2]{Cas00}
for details. The resulting tree in that case has branching points at all rational
values of the parameter $t$ (not just at the integers) and also branches of
finite length.

The topology we just described on $\cQ$ (and on $\cT$) is sometimes
called $\emph{the strong topology}$ in contrast with
a second (weaker) natural topology on $\cQ$ and $\cT$,
namely the coarsest such that for all $f\in K(X)$, the map
$v\mapsto v(f)$ is a continuous map $\cT \rightarrow \bbR$.
\end{remark}

\subsection{The Waldschmidt constant as a function on $\cQ$}

In certain cases, the invariant
$\widehat\alpha$ of valuations centered at a point $p$
of the plane is known.
We now review, following \cite{DHKRS}, what is known for quasimonomial
valuations, referring to \cite{GMM} for an overview and extension
of the results to arbitrary valuations centerd at $p$.
Fix again affine coordinates $(x,y)$ near
$\cent(v)=p=(0,0)\in\bbA^2$; for simplicity, given a series $\xi(x)$,
write
\[
\alpha(\xi,t,m) = \alpha(I_{\vv{\xi}{t},m})\ , \quad
\widehat\alpha(\xi,t)=\widehat \alpha(\vv{\xi}{t})\ , \quad
\text{and} \quad
\widehat\mu(\xi,t)=\widehat \mu(\vv{\xi}{t})\ .
\]
Recall from Exercise \ref{limits} that
$\widehat\mu(v_{\xi,t})=\widehat\alpha(v_{\xi,t})^{-1}$.
In this section we consider $\widehat\alpha$ and $\widehat\mu$
as functions of $\xi$ and $t$; it will turn out
that $\widehat\mu$ is simpler, as a function of $t$, than $\widehat\alpha$,
and we shall focus on the former.

\begin{proposition}\label{continuoust}
 For every $\xi(x)$, the function $t\mapsto \mm{\xi}{t}$,
 for $t \in [1,\infty)$, is Lipschitz continuous
 with Lipschitz constant 1.
\end{proposition}

\begin{proof}
 For every $f\in \bbC[x,y]$, the function $t\mapsto \vv[f]{\xi}{t}$
 is a tropical polynomial function of degree at most $\deg(f)$,
 as explained in remark \ref{rem:Newton}.
 Therefore, the scaled function $\mu_f:t\mapsto \vv[f]{\xi}{t}/\deg(f)$
 is continuous concave and piecewise affine linear with slopes in
 $\{0,1/\deg(f),\\2/\deg(f), \dots, 1\}$
 (compare with \cite[Corollary C]{BFJ09}).
 In particular, it is Lipschitz continuous with Lipschitz constant at most 1.

 The function $t\mapsto \mm{\xi}{t}$ in the claim is
 $\sup_{f\in \bbC[x,y]}\{\mu_f\}$; therefore it is also Lipschitz continuous
  with Lipschitz constant at most 1 (and it is not hard to see that it is actually
  equal to 1).
 \end{proof}

It is immediate to extend the definition of $\mu$ and $\widehat\mu$ to the
tree $\cT$ of all valuations centered at $p$. The continuity properties
of the resulting function $\widehat\mu:\cT\rightarrow \bbR$
---which we shall not need---
are summarized as follows:

 \begin{theorem}[Dumnicki-Harborune-Küronya-Roé-Szemberg, \cite{DHKRS}]\label{lower-semicontinuous}
    The function  $\widehat\mu:\cT\rightarrow \bbR$ is
lower semicontinuous for the weak topology and continuous
for the strong topology.
 \end{theorem}

If the series $\xi(x)$ is chosen with coefficients general enough
(see \cite{DHKRS} for details), one obtains a
 function $\mm{\xi}{t}$ that is minimal for all values of $t$:
\[\mm{\xi_{\operatorname{general}}}{t}=\min\{\mm{\xi}{t} \,|\, \xi \in \bbC[[x]]\}.\]
This minimal function, that is the same for every sufficiently general choice,
will be denoted by $\widehat\mu(t)=\mm{\xi_{\operatorname{general}}}{t}$.

\begin{corollary}[of proposition \ref{continuoust}]
  The function $\widehat\mu(t)$ is Lipschitz continuous with
Lipschitz constant 1.
 \end{corollary}

The behaviour   of the function $\widehat\mu(t)$ is known for small values and also for
square integer values of $t$, by \cite{DHKRS}. Let $F_{-1}=1$, $F_0=0$ and $F_{i+1}=F_i+F_{i-1}$ be the Fibonacci numbers,
and $\phi=(1+\sqrt{5})/2=\lim F_{i+1}/F_i$ the ``golden ratio''.

\begin{theorem}[Dumnicki-Harbourne-Küronya-Roé-Szemberg] \label{muvalues}
The continuous piecewise linear function
defined in table \ref{tablemu} agrees with $\widehat\mu(t)$ in its domain.
\begin{table}
\begin{center}
\renewcommand{\arraystretch}{1.5}
\begin{tabular}{@{}rcccc@{}}
\toprule
& \multicolumn{2}{c}{$\overbrace{\hspace{5cm}}^{i\ge 1\text{ odd}}$} \\
\multirow{2}{*}{$t\in\left[1,7+\frac19\right]\ \Bigg\{$}
& $t\in \left[\frac{F_{i}^2}{F_{i-2}^2},\frac{F_{i+2}}{F_{i-2}}\right]$
& $t\in \left[\frac{F_{i+2}}{F_{i-2}},\frac{F_{i+2}^2}{F_{i}^2}\right]$
& $t\in [\phi^4,7]$
& $t\in \left[7,\left(\frac{8}{3}\right)^2\right]$\\
& $\widehat{\mu}(t)=\frac{F_{i-2}}{F_{i}}\,t$
& $\widehat{\mu}(t)=\frac{F_{i+2}}{F_{i}}$
& $\widehat{\mu}(t)=\frac{1+t}{3}$
& $\widehat{\mu}(t)=\frac{8}{3} $ \\
\midrule
\multirow{2}{*}{$t\sim 7+\frac18\ \Bigg\{$}
& \multicolumn{2}{c}{ %
$t\in \left[\left( {\frac {24+\sqrt{457}}{17}} \right) ^{2},
 7+\frac18\right]$}
& \multicolumn{2}{c}{ %
$t\in \left[7+\frac18,\left( 24-\sqrt {455} \right) ^{2}\right]$}\\
& \multicolumn{2}{c}{ %
$\widehat{\mu}(t)=\frac{7+17\,t}{48}$}
& \multicolumn{2}{c}{ %
$\widehat{\mu}(t)=\frac{121+t}{48}$}\\
%\midrule
\multirow{2}{*}{$t\sim 7+\frac1{7+1/2}\ \Bigg\{$}
& \multicolumn{2}{c}{ %
$t\in \left[\left( {\frac {16+\sqrt {179}}{11}} \right) ^{2},
7+\frac1{7+1/2}\right]$}
& \multicolumn{2}{c}{ %
$t\in \left[7+\frac1{7+1/2},
\left( {\frac {32-\sqrt {177}}{7}} \right) ^{2}\right]$}\\
& \multicolumn{2}{c}{ %
$\widehat{\mu}(t)=\frac{7+11\,t}{32}$}
& \multicolumn{2}{c}{ %
$\widehat{\mu}(t)=\frac{121+7\,t}{64}$}\\
%\midrule
\multirow{2}{*}{$t\sim 7+\frac17\ \Bigg\{$}
& \multicolumn{2}{c}{ %
$t\in \left[\left( \frac{6+\sqrt{22}}{4}\right)^2,7+\frac17\right]$}
& \multicolumn{2}{c}{ %
$t\in \left[7+\frac17,
  \left( 12-\sqrt {87} \right) ^{2}\right]$}\\
& \multicolumn{2}{c}{ %
$\widehat{\mu}(t)=\frac{7+8\,t}{24}$}
& \multicolumn{2}{c}{ %
$\widehat{\mu}(t)=\frac{57+t}{24}$}\\
%\midrule
\multirow{2}{*}{$t\sim 7+\frac{1}{6+1/2}\ \Bigg\{$}
& \multicolumn{2}{c}{ %
$t\in \left[ \left( {\frac {20+\sqrt {218}}{13}} \right) ^{2},
7+\frac{1}{6+1/2}\right]$}
& \multicolumn{2}{c}{ %
$t\in \left[7+\frac{1}{6+1/2},
  \left(\frac{107}{40}\right)^2\right]$}\\
& \multicolumn{2}{c}{ %
$\widehat{\mu}(t)=\frac{14+13\,t}{40}$}
& \multicolumn{2}{c}{ %
$\widehat{\mu}(t)=\frac{107}{40}$}\\
%\midrule
\multirow{2}{*}{$t\sim 7+\frac15\ \Bigg\{$}
& \multicolumn{2}{c}{ %
$t\in \left[\left( \frac{8+\sqrt {29}}{5} \right) ^{2},7+\frac15\right]$}
& \multicolumn{2}{c}{ %
$t\in \left[7+\frac15,
  \left(\frac{43}{16}\right)^2\right]$}\\
& \multicolumn{2}{c}{ %
$\widehat{\mu}(t)=\frac{7+5\,t}{16}$}
& \multicolumn{2}{c}{ %
$\widehat{\mu}(t)=\frac{43}{16}$}\\
%\midrule
\multirow{2}{*}{$t\sim7+\frac14\ \Bigg\{$}
& \multicolumn{2}{c}{ %
$t\in \left[\left(\frac{35}{13}\right)^2,7+\frac14\right]$}
& \multicolumn{2}{c}{ %
$t\in \left[7+\frac14,
  \left( {\frac {35-\sqrt {877}}{2}} \right) ^{2}\right]$}\\
& \multicolumn{2}{c}{ %
$\widehat{\mu}(t)=\frac{13\,t}{35}$}
& \multicolumn{2}{c}{ %
$\widehat{\mu}(t)=\frac{87+t}{35}$}\\
%\midrule
\multirow{2}{*}{$t\sim 7+\frac12\ \Bigg\{$}
& \multicolumn{2}{c}{ %
$t\in \left[\left(\frac{4+\sqrt{2}}{2}\right)^2,7+\frac12\right]$}
& \multicolumn{2}{c}{ %
$t\in \left[7+\frac12,
  \left(\frac{22}{8}\right)^2\right]$}\\
& \multicolumn{2}{c}{ %
$\widehat{\mu}(t)=\frac{7+2\,t}{8}$}
& \multicolumn{2}{c}{ %
$\widehat{\mu}(t)=\frac{22}{8}$}\\
%\midrule
\multirow{2}{*}{$t\sim 8\ \Bigg\{$}
& \multicolumn{2}{c}{ %
$t\in \left[\left( \frac{3+\sqrt {7}}{2} \right) ^{2},8\right]$}
& \multicolumn{2}{c}{ %
$t\in \left[8, \left(\frac{17}{6}\right)^2\right]$}\\
& \multicolumn{2}{c}{ %
$\widehat{\mu}(t)=\frac{1+2\,t}{6}$}
& \multicolumn{2}{c}{ %
$\widehat{\mu}(t)=\frac{17}6$}\\
\midrule
\multicolumn{3}{c}{ %
$t=n^2, n$ an integer}
& \multicolumn{2}{c}{ %
$\widehat{\mu}(n^2)=n$}\\
\bottomrule
\end{tabular}
\end{center}
\caption{\label{tablemu}Piecewise linear function
that agrees with $\widehat\mu$ on each interval.}
\end{table}

\end{theorem}

It may be informative to look at the graphical representation of the
known behaviour of $\widehat\mu(t)$ for $t\le 9$ in figure
\ref{graph}.

\begin{figure}
 \includegraphics[width=\linewidth,keepaspectratio=true]{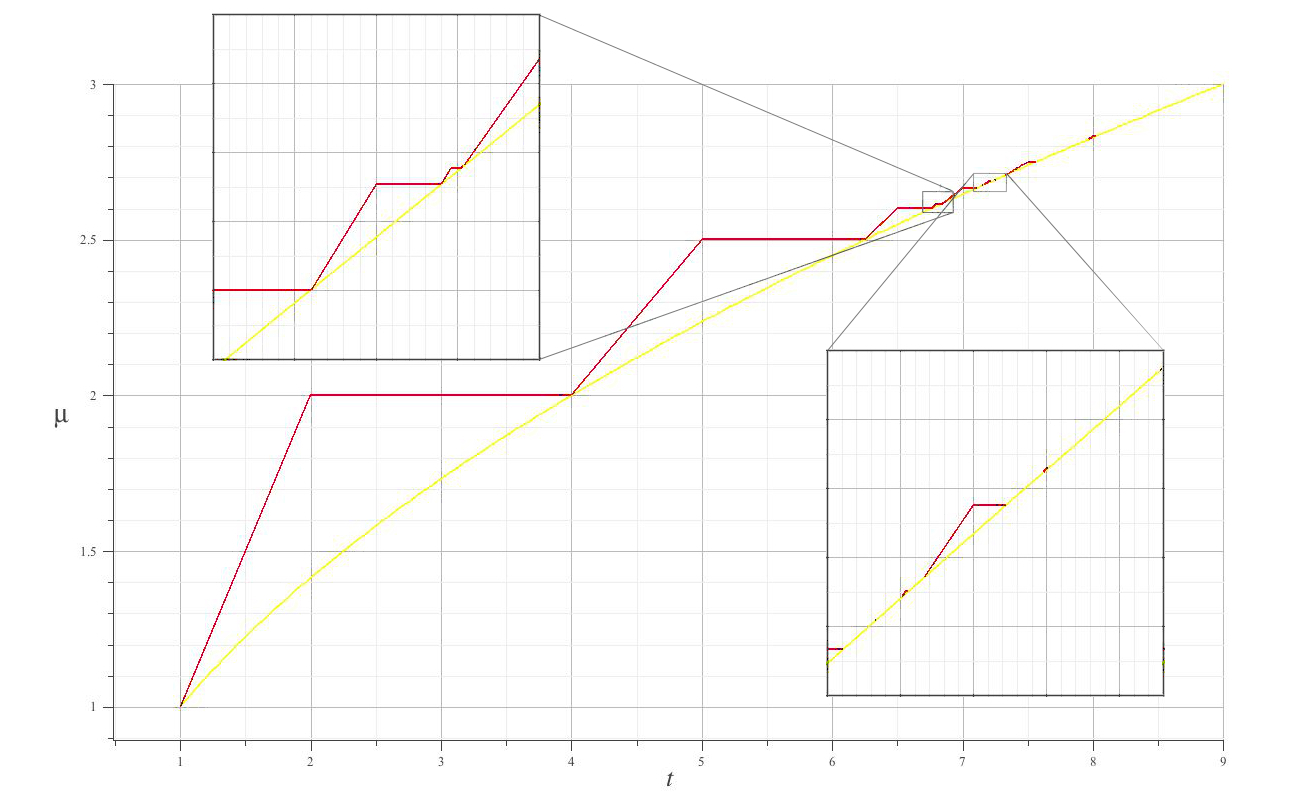}
 \caption{In red, the known behaviour of $\hat\mu(t)$ for $t\le 9$; in yellow, the lower bound
 $\sqrt{t}$. \label{graph}}
\end{figure}

\begin{remark}\label{QuadField}
At the lower endpoints of the intervals included in the left
column of table \ref{tablemu}, and at the upper endpoints of the intervals
in the right column, one has $\widehat\mu(t)=\sqrt{t}$.
Note that all such endpoints given in table
\ref{tablemu} are squares in $\mathbb{Q}$ or in the quadratic
field to which they belong.

In particular there is a sequence of \emph{rational} squares $t<8$ with
$\widehat\mu(t)= \sqrt t$, with an accumulation point at $\phi^4$;
we suspect that $\widehat\mu(t)$ can be computed for at least some rational squares
$t>9$ by existing techniques, that by continuity of
$\widehat \mu$ would allow to compute $\widehat\mu(t)$ for some nonsquare $t$.
\end{remark}

In the next section we will sketch the proof of Theorem \ref{muvalues},
and show the relationship between the partial knowledge we have
on the function
$\widehat\mu$ and the partial knowledge we have on
the Mori cone of the blown up $\bbP^2$.
At this point we already see the first analogy to Nagata's onjecture,
as the last row of table \ref{tablemu} tells us that
 for a sufficiently general choice of $\xi$, and every integer square $t=n^2$,
 the valuation $\vv{\xi}{t}$ is maximal. We will see that the connection is in fact
 stronger than an analogy: the following conjecture, put forward in \cite{DHKRS}
 implies Nagata's conjecture:

\begin{conjecture}\label{mainconj}
 For a sufficiently general choice of $\xi$, and every $t\ge 8+1/36$,
 the valuation $\vv{\xi}{t}$ is maximal.
\end{conjecture}

\subsection{The cluster of centers of a valuation}
\label{sec:cluster}

Next we are going to introduce some geometric structures attached to valuations
that allow to study $\widehat \alpha$ and $\widehat\mu$, to prove Theorem
\ref{muvalues}, and motivate the new
extension of Nagata's conjecture.

Each valuation with center at a closed point of a surface $X$
naturally determines a \emph{cluster of centers}, as follows.
To begin with, let $p_1=\cent(v) \in X$.
Consider the blowup $\pi_1:X_1\rightarrow X$
centered at $p_1$ and let $E_1$ be the corresponding exceptional divisor.
The center of $v$ on $X_1$ satisfies $\pi_1(\cent_{X_1}(v))=p_1$,
so it may only be  $E_1$ or a closed point $p_2\in E_1$.

As long as the center is a closed point, the process can be iterated:
blowing up the centers $p_1, p_2, \dots$ of $v$ either ends
with a model where the center of $v$ is an exceptional divisor $E_n$,
in which case
\[
v(f) \,=\, c\cdot \ord_{E_n}f
\]
for some constant $c$ and for every $f\in K(X)$,
and  $v$ is (still) called a divisorial valuation,
or the sequence of blowing up centers goes on indefinitely.
For each center $p_i$ of $v$, let $v_i=v(p_i)=v(E_i)$ be
the value of general curves through $p_i$

Following \cite[Chapter 4]{Cas00}, we call the collection
of points with weights
\[K=(p_1^{v_1}, p_2^{v_2},\dots),\]
the \emph{weighted cluster of points} associated to the valuation
$v$. The cluster $K$ completely determines $v$, because
for every effective divisor $D \subset X$,
\begin{equation}
\label{eq:proximity}
 v(D)=\sum_i v_i\cdot \mult_{p_i} \widetilde D_i,
\end{equation}
 where $\widetilde D_i$ is the
proper transform at $X_i$.  The sum may be infinite, but
$\widetilde D$ can have positive multiplicity at only a finite number
of centers \cite[8.2]{Cas00} (this property is not satisfied
by higher rank valuations, as explained in \cite[8]{Cas00},
but we don't consider such valuations in these notes).

\begin{exercise}
	Prove the equality \eqref{eq:proximity}.\\
	\emph{Hint}: If $\bar D$ is the total transform (pullback)
	after blowing up $p_1$,  then $v(D)=v(\bar D)$, and $$\bar D=\mult_{p_1}(D)E_1+\tilde D.$$
\end{exercise}

Sometimes we shall say that a divisor goes through
an infinitely near point to mean that its proper transform
on the appropriate surface goes through it.

\begin{example}
	Let $C\subset \bbP^2$ be a curve smooth at
	$p=(0,0)$, and $t=n\ge 1$ a natural number. Let
	$\xi\in\bbC[[x]]$ be the Taylor power series locally parameterizing
	$C$. The cluster of centers associated to the quasimonomial
	valuation $v_{\xi,t}$ is $K=(p_1^1,\dots,p_n^1)$, i.e., it
	consists of $n$ points with weights $v_i=1$,
	and the points are determined by the fact that $p_1=p$
	and $C$ goes through each $p_i$.
\end{example}

\begin{definition}
With notation as above, given indices $j<i$, the center $p_i$ is  called \emph{proximate to $p_j$}
($p_i\succ p_j$) if $p_i$ belongs to the proper transform $\widetilde E_{j}$ of the exceptional
divisor of $p_j$. Each $p_i$ with $i>0$ is proximate to $p_{i-1}$ and to
at most one additional center $p_j$,  with $j<i-1$; in this case
$p_i=\widetilde E_{{j}}\cap E_{i-1}$
and $p_i$ is called a   \emph{satellite point}. A point that is not a satellite point is called \emph{free}.
\end{definition}

\begin{remark}
\label{proximate-irrdiv}
  The irreducible components of exceptional divisors
can be computed, writing proper transforms as
combinations of total transforms, if the proximity
relations are known: $\tilde E_j=E_j-\sum_{p_i\succ p_j}E_i$.
\end{remark}

\begin{remark}\label{rm:proximity}
For every valuation $v$, and every center $p_i$ such that $v$ is not the
divisorial valuation associated to $p_i$, equation~\eqref{eq:proximity}
applied to $D=E_{j}$ gives rise to the so-called  \emph{proximity equality}
\[
v_j \,=\, \sum_{p_i \succ p_j} v_i \ .
\]

For effective divisors $D$ on $X$, the intersection number
$\widetilde D \cdot \widetilde E_{j}\ge0$
together with remark \ref{proximate-irrdiv}
yield the proximity inequality
\[
\mult_{p_j}(\widetilde D_j)\ge \sum_{p_i\succ p_j} \mult_{p_i} (\widetilde D_i)\ .
\]
\end{remark}

\begin{example}
	Let $C\subset \bbP^2$ be a curve smooth at
	$p=(0,0)$, and let $t=3/2$. Let
	$\xi\in\bbC[[x]]$ be the Taylor power series locally parameterizing
	$C$. The cluster of centers associated to the quasimonomial
	valuation $v_{\xi,t}$ is $K=(p_1^1,p_2^{1/2},p_3^{1/2})$, where
\begin{itemize}
	\item $p_1=p$,
	\item $p_2=E_1 \cap \tilde C \subset X_1$,
	\item $p_3=E_2 \cap \tilde E_1 \subset X_2$.
\end{itemize}
	In other words, the associated cluster consists of three points,
	the second of which is determined by the degree 1 coefficient of
	$\xi$, and the third is a satellite.
	
	Indeed, by definition of $v_{\xi,t}$, one has
	$v_{\xi,t}(p_1)=1$ and $v_{\xi,t}(C)=t$. For
	$t>1$, \eqref{eq:proximity} applied to $C$ means
	that $\mult_{p_2}\tilde C_1>0$, hence the point $p_2$ is
	as claimed, and in fact for $t=3/2$ one has
	\[1=v(E_1)>v(\tilde C_1)=t-1=1/2,\]
	hence $v(p_2)=1/2$. The determination of $p_3$ with its
	value follows, applying \eqref{eq:proximity} to $E_1$.
\end{example}

\subsubsection*{Valuation divisors and valuation ideals}

Assume now that  $v=\ord_{E_s}$ is the divisorial valuation
with associated cluster $K=(p_1^{v_1},\dots,p_s^{v_s})$, and let
$\pi_K:X_K\rightarrow X$ be the composition of the blowups of
all points of $K$ (in this case, $v_s=1$).
Then,  for every $m>0$,  the valuation ideal sheaf $\cI_m$ can be described as
\[
\mathcal I_m = (\pi_K)_* (\cO_{X_K}(-mE_s))\ .
\]

\begin{remark}\label{UnloadRem}
As soon as $s>1$, the negative intersection number
$-mE_s\cdot \widetilde E_{s-1}=-m$ implies that all global sections of
$\cO_{X_K}(-mE_s)$ vanish along $\widetilde E_{s-1}$,  and therefore
\[
\mathcal I_m = (\pi_K)_* (\cO_{X_K}(-mE_s-\widetilde E_{s-1})) =
(\pi_K)_* (\cO_{X_K}(-E_{s-1}-(m-1)E_s))\ .
\]
This \emph{unloads} a unit of multiplicity from $p_s$ to $p_{s-1}$.
The finite process of subtracting
all exceptional components that are met negatively, (i.e.,
starting from a divisor $D_0=-mE_s$
and successively replacing $D_i$ by $D_i-\widetilde E_j$, starting with $i=0$,
whenever $D_i\cdot \widetilde E_j<0$ for some $j$,
until one obtains a $D_i$ such that
$D_i\cdot \widetilde E_j\geq 0$ for all $j$) is
classically called \emph{unloading the weights of the cluster}.
The final uniquely determined system of weights $\bar m_i$
satisfies a relative nefness property; a divisor is
said to be \emph{nef relative to a morphism} $f$
when it intersects nonnegatively every curve mapping to a point
by $f$ \cite[1.7.11]{Laz04I}. Then
\[
D_{m} = -\sum \bar m_i E_i \ \ \text{ is nef relative to }\pi_K\ .
\]
Moreover,
\[
\mathcal I_m = (\pi_K)_* (\cO_{X_K}(D_m))
\]
and in fact general sections of $\mathcal I_m$
have multiplicity exactly $\bar m_i$ at $p_i$, and  no other
singularity. More precisely,  for any ample divisor
class $A$ on $X$, the complete system $|k(\pi_K)^*A+D_m|$
 for $k\gg0$ is base-point-free,
 its general members are smooth and they meet
each $E_j$ transversely at $\bar m_j- \sum_{p_i \succ p_j} \bar m_i$
distinct points. $D_m$ will be called \emph{valuation divisor}
because of its link with the valuation ideal sheaf.
Note that relative nefness of $D_m$ is equivalent to the proximity
inequality $\bar m_j\ge \sum_{p_i \succ p_j} \bar m_i$.

It follows using \eqref{eq:proximity} that the valuation of an effective
divisor $D$ on $X$ can be computed as a local intersection multiplicity
\[
v(D) = I_{p_1}(D,C)
\]
where $C$ is the image in $X$ of a general element of $|k(\pi_K)^*A+D_m|$.
\end{remark}
\vskip\baselineskip

\begin{exercise}\label{simple}
 Let $v=\ord_{E_s}$ be the divisorial valuation whose associated cluster is
 $K=(p_1^{v_1},\dots,p_s^{v_s})$, where $v_s=1$, and set $m_0=\sum v_i^2$.
 For every $m>0$ let
 $D_m=-\sum \bar m_i E_i$  be the unique nef divisor relative to  $\pi_K$ with
 $\mathcal I_m=(\pi_K)_* (\cO_{X_K}(D_m))$. Then
\[ D_m \le -\frac{m}{m_0} \sum v _i E_i, \]
 and equality holds when the right hand side is an integer divisor.
\end{exercise}

\begin{exercise}
	Any divisor $D$ on $X_K$ supported on $E_1, \dots, E_s$
	may be uniquely written in terms of the exceptional components:
	\[D=\sum c_i \tilde E_i.\]
	The round down of such a divisor is defined as
	\(\left\lfloor D\right\rfloor
	=\sum \lfloor c_i \rfloor \tilde E_i.\)	
Show that in the previous exercise one has
\( D_m= \left\lfloor -\frac{m}{m_0} \sum v _i E_i\right\rfloor. \)
\end{exercise}

The preceding results for $v=\ord_{E_{s}}$ readily extend to
rational valuations to give the following theorem. To state it,
let us say that a divisor on $X_K$ is \emph{contracted}
if it is supported on the exceptional divisors $E_1, \dots, E_s$.

\begin{theorem}\label{thm:valdivisor}
	Let $v$ be a rational quasimonomial valuation (i.e., assume $v(K(X))\subset \bbQ$).
	Then the associated
	cluster $K=(p_1^{v_1},\dots,p_s^{v_s})$ is finite and has rational
	weights $v_i$.
	For every $m\ge 0$ there is a unique contracted divisor $D_m$ on $X_K$,
	nef relative to $\pi_K$ and
	with $\mathcal I_m=(\pi_K)_* (\cO_{X_K}(D_m))$.
	Moreover the contracted $\bbQ$-divisor $D_v$ on $X_K$
	determined by the equalities $D_v\cdot \tilde E_i=0$
	for all $i=1, \dots, s-1$ and $D_v\cdot E_s=\frac{v_s}{\sum v_i^2}$
	(in particular $D_v$ is nef relative to $\pi_K$)
	satisfies
\[ D_m \le {m}D_v, \]
 and equality holds when the right hand side is an integer divisor.	
\end{theorem}
\begin{proof}
We refer to \cite[8.2]{Cas00} for the finiteness of the associated
cluster $K$. Then $K$ differs from the cluster associated to the
divisorial valuation $v_{E_s}$ in the multiplicative constant $v_s$
for all values, and the claims follow from the discussion above.
\end{proof}

\begin{exercise}\label{volume_from_cluster}
  Let $v$ be a divisorial valuation with associated cluster
 $K=(p_1^{v_1},\dots,p_s^{v_s})$. Then
\[
\vol (v)\,=\,\Big(\sum v_i^2\Big)^{-1}\ .
\]
\emph{Hint}: Use the \emph{codimension formula} \cite[4.7.1]{Cas00}.
\end{exercise}

Consider the group of numerical equivalence classes of $\bbR$-divisors
$N_1(X_K)$, 
and the Mori cone $\Mor(X_K)\subset N_1(X_K)$. Theorem \ref{thm:valdivisor}
allows to rephrase the definition of $\widehat\alpha$ and
$\widehat\mu$ as follows. Assume that $v$ is a rational valuation
on the projective smooth surface $X$. Then clearly
\begin{gather}\label{mu-Mor}
\widehat\alpha_D(v)=\max
\{\delta\in \bbR \,|\, \delta \pi_K^*(D)+D_v \in \Mor(X_K)\},\\
\widehat\mu_D(v)=\min
\{\epsilon\in \bbR \,|\, \pi_K^*(D)+\epsilon D_v \in \Mor(X_K)\}.
\end{gather}

\begin{exercise}\label{ex:muvol}
Prove Proposition \ref{volBKMS} for rational valuations on surfaces,
using \eqref{mu-Mor}.
\end{exercise}

 In cases when $\Mor(X_K)$ is a rational polyhedral cone,
 \eqref{mu-Mor} yields that $\widehat \mu_D (v)$ is
 a rational number, and therefore $v$ can be
 maximal only if $\sqrt{D^2/\vol(v)}$ is rational.
 In fact, all examples known of divisorial maximal valuations
 correspond to rational values of $\sqrt{D^2/\vol(v)}$,
 even for nonpolyhedral $\Mor(X_K)$.
For some examples of non-divisorial maximal valuations,
see Remark \ref{QuadField}.

 Quasimonomial valuations are exactly the valuations  whose associated cluster consists
of a few free points followed by satellites, that may be finite or infinite
in number, but not infinitely many proximate to the same center.
We are interested in \emph{very general quasimonomial valuations} on $\bbP^2$ (see \cite{DHKRS} and also \cite{GMM});
we linked the genericity condition to the coefficients of the
series $\xi$ used to define the quasimonomial valuations,
but it can be translated by saying that the free center points
of the associated cluster are general in the exceptional divisors
where they belong.

\begin{remark}\label{ContFracEx}\label{ClusterDivRem} \cite{Cas00}
The cluster $K$ of centers of $\vv{\xi}{t}$ can be easily described from
the continued fraction expansion
\[
t=n_1+\frac{1}{n_2+\frac{1}{n_3+\frac{1}{\ddots}}}\ .
\]
 $K$ consists of $s=\sum n_i$ centers; if $t=n_1$ then
they all lie on the proper transform of the germ
\[
\Gamma \colon \big\{ y = \xi(x)\big\}\  ,
\]
otherwise the first $n_1+1$ lie
on $\Gamma$ and the rest are satellites: starting from $p_{{n_1}+1}$ there
are $n_2+1$ points proximate to $p_{n_1}$, the last of which starts
a sequence of $n_3+1$ points proximate
to $p_{n_1+n_2}$ and so on. If the continued fraction is finite, with
$r$ terms, then the last $n_r$ points (not $n_r+1$) are proximate to
$p_{n_1+\dots+n_{r-1}}$.
The weights are
$$
v_i = \left\{
\begin{array}{lcl}
 1 & {\rm if} & 1\leq i \leq n_1,\\
 t-n_1 & {\rm if} & n_1+1\leq i \leq n_1+n_2,\\
v_{n_1+\dots+n_{j-1}}-n_j v_{n_1+\dots+n_{j}}& {\rm if} & n_1+\dots+n_{j}+1\leq i \leq n_1+\dots+n_{j+1}
\end{array}.
\right.
$$

If $t$ is rational, there are
only finitely many coefficients $n_1,\ldots,n_r$,
so the associated cluster %$K=(p_1^{v_1}, p_2^{v_2},\dots,p_s^{v_s})$
is finite with rational weights and the valuation is
divisorial.
If $t$ is irrational, then the sequence of centers is infinite,
the group of values has rational rank 2, and there is no surface $X_K$.

\end{remark}
\vskip\baselineskip

\begin{exercise}
\label{squares}
 If $t=n^2$ is the square of an integer, then
 a very general quasimonomial valuation
 $\vv{\xi}{t}$ is maximal.\\
 \emph{Hint}: by the generality assumption, it is enough to
 prove maximality for some choice of $\xi$. Consider a smooth
 curve of degree $n$ and its Taylor series.
\end{exercise}

\subsubsection*{Submaximal curves}
We end this section by showing how to prove Theorem \ref{muvalues}.
For all values of $t$ where $\widehat\mu(t)>\sqrt{t}$ there must
exist some curve $C$ with $m=v_{\xi,t}(C)>d \sqrt{t}$.
In other words, $\alpha(I_m)$ is smaller than expected
because of the equation $f\in I_m$ of $C$. The curve $C$ is said
to be \emph{submaximal}.

\begin{lemma}\label{negcurves}
	If there is an  irreducible
	polynomial $f\in \bbC[x,y]$  with
	\[
	\vv[f]{\xi}{t}>\frac{1}{\sqrt{\vol(\vv{\xi}{t})}}\deg(f)\ ,
	\]
	then  $\vv[f]{\xi}{t}=\mm{\xi}{t} \deg(f)$.
	
	Moreover, if $\mm{\xi}{t}>\frac{1}{\sqrt{\vol(\vv{\xi}{t})}}$,
	then there is such an irreducible
	polynomial $f$.
\end{lemma}

In the case above we  say that $f$ (or the curve $\{f=0\}$) computes
$\mm{\xi}{t}$. Since for any given $f$, the function $\vv[f]{\xi}{t}$ is
 concave and piecewise linear, the subset of $t\in \bbR$
such that $\vv[f]{\xi}{t}=\mm{\xi}{t} \deg(f)$ is always a closed
interval, and if nonempty (i.e., if $f=0$ is a submaximal
curve for some value of $t$), each endpoint of this interval
corresponds to a maximal valuation $\vv{\xi}{t}$.
Each pair of linear functions in
table \ref{tablemu} is determined by submaximal curve
that is submaximal in the union of the corresponding pair
of intervals.

\begin{proof} [Proof of Lemma\ref{negcurves}]
	By continuity of $\mm{\xi}{t}$ as a function of $t$, it
	is enough to consider the case $t\in \bbQ$. Let $v=\vv{\xi}{t}$.
	
	Let $f$ be as in the claim, and $d=\deg f$.
	It will be enough to prove that, for every polynomial $g$ with
	degree $e$ and $v(g)=w>\frac{e}{\sqrt{\vol(v)}}$, the polynomial $f$ divides $g$.
	So assume by way of contradiction that $f$ does not divide $g$,
	and compute the local intersection multiplicity
	\[I_p(f,g)=\dim_\bbC \frac{\cO_{\bbP^2,p}}{(f,g)}.\]

	Choose an integer $k$ such that
	$kw\in \bbN$ is an integer multiple of $t$, and consider the ideal
	$$I_{kw}=\{h\in \bbC[x,y]\,|\,v(h)\ge kw\}.$$
	Since obviously $g^k\in I$, the computation in Exercise
	\ref{generic-intersection} below shows that
	$$I_p(g^k,f)\ge kwv(f)>\frac{kwd}{\sqrt{t}}$$
	$$
	I_p(g,f)>\frac{wd}{\sqrt{t}}=dw\sqrt{\vol(v)}>de,
	$$
	so $f$ is a component of $g$.
	
	Now assume $\hat\mu(v)>\frac{1}{\sqrt{\vol(v)}}$. So there is a polynomial
	$g\in \bbC[x,y]$ of degree $e$ with $v(g)>\frac{e}{\sqrt{\vol(v)}}$.
	Since $v(f_1 \cdot f_2)=v(f_1)+v(f_2)$, it follows that
	at least one irreducible component $f$ of $g$,
	satisfies $v(f)>\frac{\deg f}{\sqrt{\vol(v)}}$.
\end{proof}

\begin{exercise}
	\label{generic-intersection}
	Let $f\in I_m$ and  $g\in I_n$. Then
	\[I_p(f,g)\ge\frac{\sum_{i=1}^s(mv_i)(nv_i)}
	{\left(\sum_{i=1}^s v_i^2\right)^2}.\]
	\emph{Hint}: Using linearity with respect to $m$ and $n$
	and Theorem \ref{thm:valdivisor}, reduce to the case when
	$\cI_{m}=(\pi_K)_*(\cO_{X_K}(mD_v))$. Then by
	Exercise \ref{simple} the claim is equivalent to Exercises
	4.13, 4.14 of \cite{Cas00}.
\end{exercise}

The existence of the submaximal curves needed to prove
Theorem \ref{muvalues} is due to Orevkov \cite{Ore02}.
For the intervals $[1,2]$ and $[2,4]$ (i.e., $i=1$ in the
first two rows of table \ref{tablemu}) the curve is
just the line tangent to $\{y=\xi(x)\}$. For the
intervals corresponding to $i=3$, the curve is a conic.
In general, the submaximal curve that gives the
$i$-th pair of linear functions, described in
Proposition \ref{orevkov} below, is built by applying
a sequence of Cremona transformations of degree 8
to the line tangent to $\{y=\xi(x)\}$.

\begin{exercise}
Given any power series $\xi(x)=\sum_{i\ge 1} a_i x^i$,
let $C$ be the line tangent to  $\{y=\xi(x)\}$,
namely $C:\{y-a_1 x=0\}$. Show that $C$
is submaximal for $\vv{\xi}{t}$ with $t\in (1,4)$
and compute $\mm{\xi}{t}$ in this range.
\end{exercise}

\begin{proposition}\label{orevkov}
Assume a power series $\xi(x)=\sum_{i\ge 1} a_i x^i$
is given with the coefficients $a_1,\dots,a_6$
very general.
 For each odd $i\ge 1$, there is a rational curve $C_i$
 with the following properties:
 \begin{enumerate}
 \item $\deg C_i=F_{i}$.
 \item $C_i$ has a single cuspidal singularity at $p$.
 \item The Newton polygon of its equation
 (with respect to coordinates $(x,w)$ as in remark \ref{rem:Newton})
 consists of a unique segment, with vertices
 $(0,F_{i-2})$ and $(F_{i+2},0)$.
 \end{enumerate}
 \noindent  Let $K_i$ be the weighted cluster associated
  to the valuation $\vv{\xi}{t_i}$ with $t_i=F_{i+2}/F_{i-2}$, and let
  $\pi_i:X_{K_i}\rightarrow \bbP^2$ be the blow up of all points
  of $K$. Then:
 \begin{enumerate}\addtocounter{enumi}{3}
 \item $\pi_i$ is an embedded resolution of $C_i$.
 \item The strict transform $\tilde C_i\subset X_{K_i}$
 is a $(-1)$-curve.
 \item $C_i$ is submaximal for $t$
  in the interval $\left(\frac{F_{i}^2}{F_{i-2}^2},\frac{F_{i+2}^2}{F_{i}^2}\right)$.
 \end{enumerate}
\end{proposition}

\section{Cones of b--divisors}

At the end of the preceding section it became clear that
Nagata-type statements for valuations and for extremal rays
of the Mori cone are connected, beyond simple analogy. However,
from the perspective of Mori cones, the values taken by
the Waldschmidt function $\widehat\mu$ on different parameters
$t$ (or different valuations $v_{\xi,t}$ in the valuative tree)
appear to be unrelated, as they correspond to different
blown up surfaces. In particular, the piecewise linear nature of
the known parts of the Waldschmidt function, and the quadratic
nature of its conjectural parts, show striking analogies with
the (known and conjectural) shape of the Mori cones, with no
satisfactory explanation at this point. This last section
is an attempt at giving such an explanation for the existing deep connection,
in Shokurov's language of b--divisors. This is joint work in
progress of the first author with S.~Urbinati \cite{RU17}.

\subsection{Zariski--Riemann space and b--divisors}
\emph{Birational divisors}, or simply b--divisors,
were introduced by V.~V.~Sho\-ku\-rov in the context of the Minimal Model Program,
see \cite{Sho03}. We next review some basic facts about them,
addressing the reader to \cite{BdFF}, \cite{Isk03}, \cite{Cor07} for details.

Given a normal projective variety $X$ (for our purposes, $X=\mathbb{P}^2$)
consider the set
\[
\{\pi:X_{\pi}\rightarrow X  \text{ birational morphism}\}\,/\cong
\]
of isomorphism classes
($\cong$ denotes isomorphisms $X_\pi \cong X_{\pi'}$
commuting with $\pi$ and $\pi'$)
of \emph{birational models} of $X$. This set is partially ordered
by setting $\pi_1\geq\pi_2$ if $\pi_1$ factors through $\pi_2$. This
order is inductive, i.e. any two proper birational morphisms to $X$ can be dominated by a
third one. The  \emph{Riemann-Zariski space of} $X$
is the projective limit
\[
\mathfrak{X}=\lim_{\leftarrow}\{X_\pi\rightarrow X \text{ birational morphism}\}\,/\cong
\]
 in the category of locally ringed topological spaces,
 each $X_\pi$ being viewed as a scheme with its Zariski topology
 and structure sheaf $\cO_{X_\pi}$.
 As a topological space $\mathfrak{X}$ is quasi-compact.

By a well known theorem of Zariski \cite[VI,\S17]{ZS75II} the stalks of the
structure sheaf of $\mathfrak{X}$ are exactly the
valuation rings of $K(X)$ containing $\bbC$,
so there is a natural bijection
\begin{equation*}
 \mathfrak{X} \longleftrightarrow \{ \text{valuations on } K(X)
 \text{ trivial on } \bbC\}
\end{equation*}
 The topology induced on the set of valuations
 admits as a basis of open sets the subsets of the form
 \[
 U_{f_1,\dots,f_k}=\{ v \text{ valuation such that } v(f_i)\ge 0 \, \forall i\}, \qquad \text{where }f_i\in K(X),
 \]
 or in other words, the subsets consisting of those valutations
 whose valuation rings contain a given finite subset of $K(X)$.
 The locally ringed space structure  is given by assigning to any open
 subset the intersection of the valuation rings of the valuations of the
 subset.

A Weil divisor $\overline{W}$ on $\mathfrak{X}$ is defined to be
a collection
of divisors $W_\pi\in \Div (X_\pi)$, one on each birational model
$\pi: X_\pi\rightarrow X$, compatible under
push-forward, that is,  $\mu_*W_{\pi} = W_{\pi'}$
if $\pi=\mu\circ\pi'$.
The element $W_\pi$ of the collection $\overline{W}$
is called \emph{trace} of $\overline{W}$ on $X_\pi$.
The group of Weil divisors on $\mathfrak{X}$ is therefore
$$
\Div (\mathfrak{X})= \lim_{\leftarrow}\{\Div (X_\pi)\}
$$
where the arrow refers to push--forwards of the divisors.

On the other hand, a Cartier divisor $\overline D$  on $\mathfrak{X}$
is a Weil divisor for which there is a model $X_0$ such that for every
other model $X_{\pi}$ dominating $X_0$, the trace
$D_\pi$ of $\overline D$ on $X_\pi$ is the pull-back of the trace $D_0$.
Thus the group of Cartier divisors $\mathfrak{X}$
is also a limit of groups of Cartier divisors, but under
pullbacks rather than pushforwards.
$$
\CDiv (\mathfrak{X})= \lim_{\rightarrow}\{\CDiv (X_n)\}.
$$
In particular, a Cartier divisor $D_\pi$ on a model $X_\pi$ defines a Cartier divisor  $\overline{D}$ on $\mathfrak{X}$, by pulling back $D_\pi$ on all models dominating $X_\pi$ and pushing forward on all other models. $D_\pi$ is called a \emph{determination} of $\overline{D}$.
In other words, there is an injection $CDiv(\mathfrak{X})\hookrightarrow Div(\mathfrak{X})$, due to the fact that $\pi_*\pi^*(D)=D$  when $\pi$ is a birational map.
Cartier and Weil divisors on $\mathfrak{X}$ are called  \emph{b--divisors} of $X$, to recall that they are divisors up to birational equivalence.
We set
$\Div_{\mathbb{R}}(\mathfrak{X})=\Div(\mathfrak{X})\otimes \mathbb{R}$ and $\CDiv_{\mathbb{R}}(\mathfrak{X})=\CDiv(\mathfrak{X})\otimes \mathbb{R}$  the $\mathbb{R}$--Weil b--divisors and the $\mathbb{R}$--Cartier b--divisors respectively.

Since nefness and bigness are stable under pullbacks by birational
morphisms, we can refer to nefness and bigness of Cartier b--divisors.
Since the valuation of a divisor is preserved by pullback,
$v(\overline{D})$ is well defined for every valuation $v$ and
every Cartier b--divisor $\overline{D}$. In the case of
a divisorial valuation $v$, one can even define the valuation
of a Weil b--divisor, as follows. Let $X_\pi$ be a model
in which there is a prime divisor $E\in\CDiv(X)$ with $v=t\cdot \ord_E$
for some $t\in \bbR$; then for every b--divisor $\overline{W}$,
set $v(\overline{W})=(t\cdot \ord_E)(\overline{W})$ to be equal to
$t$ times the coefficient of $E$ in $W_\pi$.
A b--divisor can therefore be interpreted as a function
$v\mapsto v(\overline{W})$ on the set $\mathcal{V}$ of divisorial
valuations of $X$.
Since distinct b--divisors clearly give distinct functions,
we obtain an immersion
$$
\Div_{\mathbb{R}}(\mathfrak{X})\hookrightarrow \mathbb{R}^{\mathcal{V}}
=\Pi_{v\in \mathcal{V}}\mathbb{R}=\operatorname{func}(\mathcal{V},\bbR) ,
 $$
that is then used to endow the set of b--divisors  with
the topology induced by the topology of pointwise convergence on
 $\mathbb{R}^\mathcal{V}$;
this is called the \emph{topology of  coefficent--wise convergence} on
$\Div_{\mathbb{R}}(\mathfrak{X})$,
for which $\lim_j \overline{W}_j = \overline{W}$ if and only if $\lim_j v_{E_\pi}(W_j)_\pi = v_{E_\pi}(W_\pi)$ for each
prime divisor $E_\pi$ on the model $X_\pi \rightarrow X$.

\medskip

One can also consider the group of Cartier b--divisors modulo numerical equivalence, defining
the Neron--Severi space
$$
N^1(\mathfrak{X})\otimes \mathbb{R}=N^1(\mathfrak{X})_\mathbb{R}=\lim_\rightarrow N^1(X_\pi)
$$
where the maps defining the projective limit are given by pulling back:
a class is determined by the class of a Cartier divisor
in some blow up of $X$.

Assume from now on that $X$ is a surface.
In that case, the group of 1--dimensional numerical classes
of $\mathfrak{X}$ is
$$
N_1(\mathfrak{X})=\lim_\leftarrow N^1(X_\pi),
$$
here the maps are given by push--forward (on arbitrary dimension,
the $(n-1)$--dimensional numerical classes are defined by the projective
limit on the smooth models). The limit topology on $N^1(\mathfrak{X})$
and $ N_1(\mathfrak{X})$ is compatible with the topology of
coefficient-wise convergence defined above for $\Div_\bbR{\mathfrak X}$.

There is  a natural injection $N^1(\mathfrak{X})\hookrightarrow N_1(\mathfrak{X})$, by identifying a class $\beta\in N^1(\mathfrak{X})$ to the class $\overline{\beta}\in N_1(\mathfrak{X})$ determined by
pulling back $\beta$ on all higher models: by definition, the injection is continuous in the projective limit topology and $N^1(\mathfrak{X})$ is dense in $N_1(\mathfrak{X})$ \cite[1.9]{BdFF}.

\subsubsection*{Relative Zariski decomposition}
Zariski ---in what can be considered
a foundational work of the asymptotic theory of linear systems---
showed in \cite{Zar62} that
any effective divisor \(D\) on a smooth surface
can be decomposed as a sum of a \emph{positive} (nef,
accountable for all sections in $H^0(X,mD)$)
and an effective \emph{negative} part (whose multiples are
a fixed part in all multiples of the linear series $D$)
with the following properties:

\begin{theorem}[Zariski decomposition]
Every pseudoeffective \(\bbQ\)-divisor \(D\) on a smooth surface $X$
admits a unique decomposition
\(
D=P+N,
\)
where \(P\) is a nef \(\bbQ\)-divisor, \(N\) is an effective
\(\bbQ\)-divisor, and if $N$ is nonzero then
the components \(N_i\) of \(N\) have negative definite intersection matrix,
and \(P\cdot N_i=0\).
\end{theorem}
The generalization to pseudoeffective \(\bbQ\)-divisors is due to Fujita \cite{Fuj79}. We refer to \cite{Bad01}  and to the more recent and nice \cite{Bauer09} for a proof.
The Zariski decomposition is a most powerful tool;
from our viewpoint, since $\oplus _{m\geq0} H^0(X,mD)= \oplus _{m\geq0} H^0(X,\lfloor mP\rfloor)$, Zariski decomposition allows us to reduce the question of finite generation of $\oplus _{m\geq0} H^0(X,mD)$ to the case where $D$ is nef.
On the other hand, in the previous section, the notion of nef divisor
relative to a morphism (more specifically, relative to the blow up morphism
of the points of a cluster) became important to study the valuation
ideals of a rank 1 valuation. The notion of relative nefness naturally
leads to a notion of \emph{relative} Zariski decomposition, that
has been considered in the literature \cite{Mor86}, \cite{CS93}
in more general settings and
for different purposes. The version most useful for us is the following.

\begin{theorem}[Relative Zariski decomposition]\label{thm:relzariski}
	Let $\pi:X_\pi\rightarrow X$ be a birational morphism of
	smooth surfaces. Every \(\bbQ\)-divisor \(D\) on $X_\pi$
	admits a unique decomposition
	\(
	D=P_\pi+N_\pi,
	\)
	where \(P_\pi\) is a \(\bbQ\)-divisor nef relative to $\pi$,
	\(N_\pi\) is an effective \(\bbQ\)-divisor with $\pi_*(N_\pi)=0$,
	and if $N_\pi$ is nonzero then
	the components \(N_i\) of \(N_\pi\) have negative definite intersection matrix, and \(P\cdot N_i=0\).
\end{theorem}

If $D$ is pseudoeffective, then the relative $D=P_\pi+N_\pi$
and absolute $D=P+N$ Zariski
decompositions are related; $N$ (respectively $N_\pi$)
is the smallest effective  $\mathbb{Q}$--divisor such that $P=D-N$
(respectively $P_\pi=D-N_\pi)$ is nef (respectively nef
relative to $\pi$); therefore
$N_\pi\le N$. However, the relative version
is far easier to prove!

\begin{exercise}
	Prove Theorem \ref{thm:relzariski}.\\
	\emph{Hint}: Let $E_1, \dots, E_n$ be the finite set of curves
	contracted by $\pi$. You can use the well-known fact that the
	intersection matrix of $(E_1,\dots,E_n)$ is negative definite
	to solve for the coefficients of $N_\pi=a_1E_1+\dots+a_nE_n$.
\end{exercise}

\begin{example}
	Let $X=\bbP^2$, and $L\subset \bbP^2$ a line.
	Given be a divisorial valuation $v$ on
	$K(X)$, with associated weighted cluster
	$K=(p_1^{v_1},\dots,p_s^{v_s})$, let
	$\pi_K:X_K\rightarrow \bbP^2$ be the blowup of all points in
	$K$, and let $E_1, \dots, E_s$ be the (total transforms of the)
	exceptional divisors. Then positive part of
	the Zariski decomposition of $-E_s$
	is the divisor $D_v$ of Theorem \ref{thm:valdivisor}.
	Therefore, for every $\delta$, the Zariski decomposition of
	$\delta \pi_K^*(L)-mE_s$ has positive part
	$\delta \pi_K^*(L)+mD_v$
\end{example}
Zariski decompositions are preserved by pullbacks,
because nefness is,
so it is natural to ask about
a Zariski-type decomposition for b--divisors.
A b--divisor $\overline{P}$ on X  is b--nef if there is a
determination $P_\pi$ of $\overline{P}$ on a model
$\pi: X_\pi \rightarrow X$ such that $P_\pi$ is nef.
This question has been addressed by A.~Küronya and C.~Maclean in
\cite{KM08}, showing that such a decomposition exists
for b--divisors on (normal) varieties of arbitrary dimension.

\begin{theorem}[Zariski decomposition for b--divisors, \cite{KM08}]
  Let $X$ be a smooth projective surface, $D$  an effective
  $\mathbb{Q}$--b--divisor on $X$. There is a unique decomposition
  $D = P +N$, where $P, N$ are effective $\mathbb{Q}$--b--divisors,
  such that $ H^0(X,\lfloor mD\rfloor)= H^0(X,\lfloor mP_D\rfloor)$, and
$P_D$ is a limit of b-nef b--divisors on every proper birational model $Y\rightarrow X$, and for any nef b--divisor $P'\leq D$ the inequality $P'\leq P_D$ holds.
\end{theorem}

\begin{remark} Given a b--divisor $\overline{D}$, the associated b-divisorial sheaf $\cO_X(\overline{D})$ is defined on an open subset $U$ by $
\Gamma(U,\mathcal{O}_X(\overline{D}))
= \{\varphi\in  K(X) | (div_X\varphi + \overline{D})|_U \geq 0\} $
(see \cite{KM08} or \cite{BFJ09}).
It is not a coherent sheaf, but there is  a natural  inclusion
$H^0(X, \mathcal{O}_X(\overline{D}))\hookrightarrow H^0(X, \mathcal{O}_X({D}))$ thus
$H^0(X, \mathcal{O}_X(\overline{D}))$ is   finite-dimensional.

The positive part of $\mathbb{Q}$-b-divisor $\overline{D}$ on $X$ in the theorem is
$$
\overline{P}_D= \max\{\overline{P} | \overline{P} \text{ a nef $\mathbb{Q}$--b--divisor }, \overline{P}\leq \overline{D}\}.
$$

\end{remark}

\subsection{Waldschmidt function through cones in $N^1(\mathfrak{X})_\bbR$}
Fix the origin  point $p_1$ on $\mathbb{P}^2$,
and affine coordinates $x,y$ around it. For every power series
$\xi \in \bbC[[x]]$ and every real number $t\ge1$ consider the
valuation $\vv{\xi}{t}$ defined in section \ref{space_of_valuations}.
Whenever $t\in \bbQ$, Theorem \ref{thm:valdivisor} provides an
associated cluster $K=(p_1^{v_1},\dots,p_s^{v_s})$
and a relatively nef divisor $D_{\vv{\xi}{t}}$,
which we have shown to be the positive part of the relative Zariski
decomposition of $-v_sE_s$.
Set $D_{\xi,t}$ the Cartier b--divisor on $\mathfrak{X}$
defined by $D_{\vv{\xi}{t}}$.

We are now ready to give an alternative proof of Proposition
\ref{continuoust},
which relates the continuity of the Waldschmidt function with
the closed convex nature of the Mori cone.

We will consider the map $\operatorname{div}$ extended by continuity to
$[0,\infty)\times[1,\infty)$.
\begin{proposition}
Let $L$ be the class in $N^1(\mathfrak{X})_\bbR$ of a line in
$\bbP^2$, and fix a series $\xi\in \bbC[[x]]$.
For every rational $t$, set $D_{\xi,t}$
the associated Cartier b--divisor. The function
\begin{align*}
\operatorname{div}:\bbR_{\ge 0} \times \bbQ_{\ge1}& \longrightarrow N^1(\mathfrak{X})_\bbR\\
(a,t)& \longmapsto aL-D_{\xi,t}
\end{align*}
is continuous.
\end{proposition}

\begin{proof}
	The topology of $N^1(\mathfrak{X})_\bbR$
	is induced by the topology of coefficientwise convergence,
	so it is enough to observe that for every prime divisor $D_\pi$
	on every model $X_\pi$ the map $t\mapsto\vv{\xi}{t}(D_\pi)$
	is continuous.
\end{proof}

\begin{proof}[Alternative proof of Theorem \ref{continuoust} (sketch)]
	We will prove that $\widehat\alpha$ is continuous (for the
	strong topology).
	
	The Mori cone $\Mor(\mathfrak{X})_\bbR$
	in $N^1(\mathfrak{X})_\bbR$ is a closed convex cone.
	Because it is closed, and $\operatorname{div}$ is continuous,
	its preimage $\operatorname{div}^{-1}(\Mor(\mathfrak{X})_\bbR)$
	is a closed set. Therefore
	\[
	\widehat\alpha(\xi,t)=\min\{a\,|\,\operatorname{div}(a,t)\in
	\Mor(\mathfrak{X})_\bbR\}
	\]
	is lower semicontinuous as a function of $t$.
	Then using that $\Mor(\mathfrak{X})_\bbR$ is convex,
	it follows that $\widehat\alpha$ is actually continuous.
\end{proof}

\bibliographystyle{plain}
{\footnotesize\bibliography{BibNagata}{}

\noindent Joaquim Roé
\\Universitat Autònoma de Barcelona, Departament de Matemàtiques
\\08193 Bellaterra (Barcelona) Spain. \url{jroe@mat.uab.cat}.
\\ \\Paola Supino,
\\Università degli Studi Roma Tre, Dipartimento di Matematica e Fisica, Sezione di Matematica,
\\Largo San Leonardo Murialdo 1,
00146, Roma, Italy.
\url{supino@mat.uniroma3.it}}

\end{document}